\documentclass[12pt]{amsart}

\usepackage{epsf}
\usepackage{graphicx}
\usepackage{latexsym}
\usepackage{amsfonts}
\usepackage{amscd}
\usepackage{amssymb}
\usepackage{amsmath}
\usepackage{amsthm}
\usepackage{bm}
\usepackage[lmargin=1in,rmargin=1in,tmargin=1.25in,bmargin=1in,dvips]{geometry}
\usepackage{color}
\usepackage{pinlabel}
\usepackage{paralist}
\usepackage{enumitem}
\usepackage[all]{xy}
\usepackage[unicode, linktocpage, bookmarksnumbered=true,backref=true]{hyperref}

\newtheorem{thm}{Theorem}

\newtheorem{conj}{Conjecture}
\newtheorem{cor}[thm]{Corollary}

\newtheorem{theorem}{Theorem}[section]
\newtheorem{lemma}[theorem]{Lemma}

\newtheorem{proposition}[theorem]{Proposition}

\theoremstyle{definition}
\newtheorem{definition}[theorem]{Definition}

\def\R{\mathbb{R}}

\def\Z{\mathbb{Z}}

\def\F{\mathbb{F}}

\def\HF {\widehat{\operatorname{HF}}}
\def\CFK {\widehat{\operatorname{CFK}}}
\def\HFK {\widehat{\operatorname{HFK}}}

\def\HFL {\widehat{\operatorname{HFL}}}

\def\CFD {\widehat{\operatorname{CFD}}}
\def\CFA {\widehat{\operatorname{CFA}}}
\def\CFAA {\widehat{\operatorname{CFAA}}}
\def\CFDD {\widehat{\operatorname{CFDD}}}
\def\CFAD {\widehat{\operatorname{CFAD}}}
\def\CFDA {\widehat{\operatorname{CFDA}}}

\def\CFKm {\operatorname{CFK}^-}

\def\Kappa {\mathrm{K}}

\def\Mu {\mathrm{M}}

\def\AA {\mathcal{A}}
\def\BB {\mathcal{B}}

\def\FF {\mathcal{F}}

\def\II {\mathcal{I}}

\def\I{\mathbb{I}}

\def\RR{\mathfrak{R}}
\def\SS{\mathfrak{S}}

\newcommand{\abs}[1] {\left\lvert #1 \right\rvert}

\def\Th{^{\text{th}}}
\def\minus{\smallsetminus}
\def\co{\colon\thinspace}
\def\one{\bm{1}}

\DeclareMathOperator{\im}{im} \DeclareMathOperator{\id}{id} \DeclareMathOperator{\rank}{rank}
  \DeclareMathOperator{\pt}{pt}
 
\DeclareMathOperator{\Span}{Span} 
 \DeclareMathOperator{\gr}{gr}
\DeclareMathOperator{\first}{first} \DeclareMathOperator{\last}{last}

\numberwithin{equation}{section}

\begin{document}

\title{Splicing knot complements and bordered Floer homology}

\author{Matthew Hedden}
\address{Department of Mathematics \\ Michigan State University \\ 619 Red Cedar Road \\ East Lansing, MI 48824}
\email{mhedden@math.msu.edu}

\author{Adam Simon Levine}
\address{Department of Mathematics \\ Princeton University \\ Fine Hall, Washington Road \\ Princeton, NJ 08544}
\email{asl2@math.princeton.edu}

\thanks{Matthew Hedden gratefully acknowledges support from NSF grant DMS-0906258, NSF CAREER grant DMS-1150872, and an Alfred P. Sloan Research Fellowship. Adam Simon Levine was partially supported by NSF Postdoctoral Research Fellowship grant DMS-1004622.}

\begin{abstract}
We show that the integer homology sphere obtained by splicing two nontrivial knot complements in integer homology sphere L-spaces has Heegaard Floer homology of rank strictly greater than one.  In particular, splicing the complements of nontrivial knots in the $3$-sphere never produces an L-space. The proof uses bordered Floer homology.
\end{abstract}

\maketitle

\section{Introduction}

A rational homology $3$-sphere $Y$ is called an \emph{L-space} if the rank of its Heegaard Floer homology group $\HF(Y)$ equals the order of $H_1(Y;\Z)$. Examples of L-spaces include $S^3$, lens spaces, all manifolds with finite fundamental group \cite[Proposition 2.3]{OSzLens}, and the branched double covers of alternating (or, more generally, quasi-alternating) links in $S^3$ \cite[Proposition 3.3]{OSzDouble}. Since the rank of $\HF(Y)$ is always greater than or equal to $\abs{H_1(Y;\Z)}$ \cite[Proposition 5.1]{OSzProperties}, L-spaces are the manifolds with the smallest possible Heegaard Floer homology, and it is natural to ask for a complete classification of L-spaces or a more topological characterization \cite[Question 11]{OSzDiagrams}. The following conjecture, first raised by Ozsv\'ath and Szab\'o \cite[p.~40]{OSzLectures}, is of central importance to Heegaard Floer theory:

\begin{conj} \label{conj:S3}
If $Y$ is an irreducible homology sphere that is an L-space, then $Y$ is homeomorphic to either $S^3$ or the Poincar\'e homology sphere. \end{conj}

Thus, the conjecture asserts that the classification of L-spaces with the singular homology of the $3$-sphere is extremely simple: they are simply the connected sums of zero or more copies of the Poincar\'e sphere (with either orientation).
Conjecture \ref{conj:S3} is known to hold for manifolds obtained by Dehn surgery on knots in $S^3$ \cite[Proof of Corollary 1.3]{OSzGenus}, \cite[Proof of Corollary 1.5]{GhigginiFibred} and for all Seifert fibered spaces \cite{RustamovLspaces}. In light of the Geometrization Theorem \cite{Perelman1,Perelman2,MorganTian}, one should consider how Heegaard Floer homology behaves under the operation of gluing along incompressible tori. The following conjecture would reduce Conjecture \ref{conj:S3} to the case of hyperbolic $3$-manifolds:

\begin{conj} \label{conj:torus}
If $Y$ is an irreducible homology sphere that contains an incompressible torus, then $Y$ is not an L-space.
\end{conj}

\noindent The purpose of this paper is to prove a special case of Conjecture \ref{conj:torus}.

To describe our result, let the exterior of a knot $K$ in a homology sphere $Y$ be denoted by $X_K$. The meridian and Seifert longitude of $K$, viewed as curves in $\partial X_K$, are respectively denoted $\mu_K$ and $\lambda_K$. Given knots $K_1 \subset Y_1$ and $K_2 \subset Y_2$, let $Y(K_1, K_2)$ denote the manifold obtained by gluing $X_{K_1}$ and $X_{K_2}$ via an orientation-reversing diffeomorphism $\phi\co \partial X_{K_1} \to \partial X_{K_2}$ taking $\lambda_{K_1}$ to $\mu_{K_2}$ and $\lambda_{K_2}$ to $\mu_{K_1}$. We say that $Y(K_1, K_2)$ is obtained by \emph{splicing} the knot complements $X_{K_1}$ and $X_{K_2}$. The Mayer--Vietoris sequence shows that $Y(K_1,K_2)$ is a homology sphere. The image of $\partial X_{K_1}$ is incompressible in $Y(K_1, K_2)$ if and only if the knots $K_1$ and $K_2$ are both nontrivial. Furthermore, a separating torus $T$  in a homology sphere $Y$ canonically determines a decomposition $Y = Y(K_1, K_2)$: if $Y = X_1 \cup_T X_2$, we obtain $Y_1$ (resp.~$Y_2$) by Dehn filling $X_1$ (resp.~$X_2$) along the unique slope in $T$ that bounds a surface in $Y_2$ (resp.~$Y_1$), and we let $K_1$ (resp.~$K_2$) be the core of the glued-in solid torus.

The main result of this paper is the following:

\begin{thm} \label{thm:main}
Let $Y_1$ and $Y_2$ be L-space homology spheres, and let $K_1 \subset Y_1$ and $K_2 \subset Y_2$ be nontrivial knots.
Then $\dim \HF(Y(K_1, K_2)) > 1$.
\end{thm}

\noindent Removing the hypothesis that $Y_1$ and $Y_2$ are themselves L-spaces would complete the proof of Conjecture \ref{conj:torus}.  Of course we have the immediate corollary:

\begin{cor} \label{cor:S3}
Splicing the complements of nontrivial knots in the $3$-sphere never produces an $L$-space.
\end{cor}

Our strategy for studying $\HF(Y(K_1, K_2))$ is to relate it to the knot Floer homology of $K_1$ and $K_2$. For a knot $K \subset Y$ in an integral homology sphere, $\HFK(Y,K)$ is a bigraded vector space over $\F=\Z/2\Z$,
\[
\HFK(Y,K) = \bigoplus_{a,m \in\Z} \HFK_m(Y,K,a),
\]
whose graded Euler characteristic is the Alexander polynomial of $K$ \cite{OSzKnot, RasmussenThesis}. These groups detect the Seifert genus of $K$ \cite[Theorem 1.2]{OSzGenus}, in the sense that
\[
g(K) = \max \{a \mid \HFK_*(Y,K,a) \ne 0\} = -\min \{a \mid \HFK_*(Y,K,a) \ne 0\}.
\]
If $K_1$ and $K_2$ are nontrivial knots, we show that $\HF(Y(K_1,K_2))$ contains a subspace of dimension
\[
2 \cdot \dim \HFK_*(Y_1, K_1, -g(K_1)) \cdot \dim \HFK_*(Y_2, K_2, -g(K_2)) \ge 2,
\]
which implies Theorem \ref{thm:main}. Indeed, since $\dim \HFK_*(Y,K,-g(K))=1$ if and only if $K$ is a fibered knot \cite{GhigginiFibred, NiFibered}, we obtain a stronger lower bound on $\dim \HF(Y(K_1,K_2))$ if either $K_1$ or $K_2$ is non-fibered.

Our basic tool for proving Theorem \ref{thm:main} is \emph{bordered Floer homology} \cite{LOTBordered}, which can be used to compute the Heegaard Floer homology of a closed $3$-manifold obtained by gluing two pieces along a common boundary as the homology of the derived tensor product of algebraic invariants associated to the pieces. We review some of the basics of this theory in Section \ref{sec:bordered}. In the present setting, we have
\[
\HF(Y(K_1, K_2)) \cong H_*(\CFA(X_{K_1}) \boxtimes \CFD(X_{K_2})),
\]
where $\CFA(X_{K_1})$ and $\CFD(X_{K_2})$ are the bordered invariants of $X_{K_1}$ and $X_{K_2}$ with suitable boundary parameterizations. Lipshitz, Ozsv\'ath, and Thurston give a formula describing $\CFD$ of the complement of a knot in an L-space homology sphere in terms of the knot Floer complex of the knot \cite{LOTBordered}, and a simple algorithm (given below as Theorem \ref{thm:identity}) yields a similar description of $\CFA$. Using an Alexander grading on the bordered invariants, we can identify subspaces of $\CFA(X_{K_1})$ and $\CFD(X_{K_2})$ that are isomorphic to the corresponding knot Floer homology groups in extremal Alexander grading and whose algebraic structure can be understood quite explicitly. These subspaces combine in the tensor product to produce the subgroup of $\HF(Y(K_1, K_2))$ described above.

In a sequence of preprints in 2008, Eaman Eftekhary announced a proof of Conjecture \ref{conj:torus}, but several delicate technical issues were overlooked in his original treatment. In \cite{EftekharySplicing}, Eftekhary provides a chain complex that ostensibly computes $\HF(Y(K_1,K_2))$ in terms of data associated to $K_1$ and $K_2$, essentially by using a precursor to bordered Floer homology. The original version of this complex yielded incorrect results; for instance, its homology has rank $13$ in the case where $K_1$ and $K_2$ are both the right-handed trefoil in $S^3$, whereas a computation using bordered Floer homology, given below in Section \ref{sec:examples}, shows that the correct rank is only $7$.\footnote{In this case, $Y(K_1,K_2)$ can also be obtained as $+1$ surgery on the positive, untwisted Whitehead double of the right-handed trefoil, whose knot Floer complex is known via \cite{HeddenWhitehead, HeddenLivingstonRuberman}. The surgery formula from \cite{OSzKnot} confirms that $\HF(Y(K_1,K_2))$ has rank $7$.} Subsequent to the submission of the present article, Eftekhary released a revision of \cite{EftekharySplicing} that provides a corrected version of this chain complex  and apparently yields an alternate proof of Theorem \ref{thm:main}. (However, Eftekhary's original proof of Conjecture \ref{conj:torus} relies on work that has been retracted.)

\subsection*{Acknowledgments}
The authors are grateful to Eaman Eftekhary, Jonathan Hanselman, Jen Hom, Robert Lipshitz, Peter Ozsv\'ath, and Dylan Thurston for many enlightening conversations, and to the referees for helpful suggestions.

\section{Bordered Heegaard Floer homology} \label{sec:bordered}

We begin by reviewing a few basic definitions and facts regarding bordered Heegaard Floer homology \cite{LOTBordered}, focusing on the case of manifolds with torus boundary. Some of this material is adapted from the second author's exposition in \cite[Section 2]{LevineDoublingOperators}.

\subsection{Algebraic preliminaries}
In this subsection, we recall the key algebraic structures that occur in bordered Floer homology, known as {\em $\AA_\infty$--modules} and {\em type D structures}. While these objects can be defined in general over an underlying $\AA_\infty$--algebra $\AA$, the relevant algebra for our purposes is merely differential graded, so it will be convenient to give the definitions in this simplified setting.

Let $(\AA, d)$ be a unital differential algebra over $\F = \Z/2\Z$, and assume that the set  $\II$ of idempotents in $\AA$ is a commutative subring of $\AA$ and possesses a basis $\{\iota_i\}$ over $\F$ such that $\iota_i \iota_j = \delta_{ij} \iota_i$ and $\sum_i \iota_i = \one$, the identity element of $\AA$. A \emph{(right) $\AA_\infty$--module} or \emph{(right) type $A$ module} over $\AA$ is a vector space $M$ equipped with a right action of $\II$ such that
\[
M = M \iota_1 \oplus...\oplus M \iota_n
\]
as a vector space, and multiplication maps
\[
m_{k+1}\co M \otimes_\II \underbrace{\AA \otimes_\II \dots \otimes_\II \AA}_{k
\text{ times}} \to M
\]
satisfying the $\AA_\infty$ relations: for any $x \in M$ and $a_1, \dots, a_n \in \AA$,
\begin{equation} \label{eq:a-relation}
\begin{split}
0 & = \sum_{i=0}^n m_{n-i+1}(m_{i+1}(x \otimes a_1 \otimes \dots \otimes a_i) \otimes a_{i+1} \otimes \dots \otimes a_n) \\
& + \sum_{i=1}^{n} m_{n+1}(x \otimes a_1 \otimes \dots \otimes a_{i-1} \otimes d(a_i) \otimes a_{i+1} \otimes \dots \otimes a_n) \\
& + \sum_{i=1}^{n-1} m_n(x \otimes a_1 \otimes \dots \otimes a_{i-1} \otimes a_i a_{i+1} \otimes a_{i+2} \otimes \dots \otimes a_n).
\end{split}
\end{equation}
We also require that $m_2(x \otimes \one)=x$ and $m_k(x \otimes \dots \otimes \one \otimes \cdots)=0$ for $k>2$. We say that $M$ is \emph{bounded} if $m_k=0$ for all $k$ sufficiently large.

A \emph{(left) type $D$ structure} over $\AA$ is an $\F$-vector space $N$, equipped with a left action of $\II$ such that
\[
N = \iota_1 N \oplus \cdots \oplus \iota_n N,
\]
and a map
\[
\delta_1\co N \to \AA \otimes_\II N
\]
satisfying the relation
\begin{equation} \label{eq:d-relation}
(\mu \otimes \id_N) \circ (\id_\AA \otimes \delta_1) \circ \delta_1 + (d \otimes \id_N) \circ \delta_1 = 0,
\end{equation}
where $\mu\co \AA \otimes \AA \to \AA$ denotes the multiplication on $\AA$. If $N$ is a type $D$ module, the tensor product $\AA \otimes_\II N$ is naturally a left differential module over $\AA$, with module structure given by $a \cdot (b \otimes x) = ab \otimes x$, and differential $\partial(a \otimes x) = a \cdot \delta_1(x)+d(a)\otimes x$. Condition \eqref{eq:d-relation} translates to $\partial^2=0$. We inductively define maps
\[
\delta_k\co N \to \underbrace{\AA \otimes_\II \dots \otimes_\II \AA}_{k \text{times}} \otimes_\II N
\]
by $\delta_0 = \id_N$ and $\delta_k = (\id_{\AA^{\otimes k-1}} \otimes \delta_1) \circ \delta_{k-1}$. We say $N$ is \emph{bounded} if $\delta_k=0$ for all $k$ sufficiently large.

If $M$ is a type $A$ module and $N$ is a type $D$ module, the $\AA_\infty$--tensor product $M \mathbin{\tilde\otimes} N$ \cite[Definition 2.12]{LOTBordered} is a chain complex whose chain homotopy type depends only on the chain homotopy types of $M$ and $N$ (using suitable notions of chain homotopy equivalence for type $A$ and $D$ modules).



If either $M$ or $N$ is bounded, the \emph{box tensor product} $M \boxtimes N$ is the vector space $M \otimes_\II N$, equipped with the differential
\[
\partial^\boxtimes(x \otimes y) = \sum_{k=0}^\infty (m_{k+1} \otimes \id_N)(x \otimes \delta_k(y)).
\]
This is a finite sum, and \eqref{eq:a-relation} and \eqref{eq:d-relation} imply that $\partial^\boxtimes \circ \partial^\boxtimes = 0$. Lipshitz, Oszv\'ath, and Thurston \cite[Proposition 2.34]{LOTBordered} show that when $M$ or $N$ is bounded, $M \boxtimes N$ is chain homotopy equivalent to $M \mathbin{\tilde \otimes} N$.

In addition to type $A$ modules and $D$ structures over $\AA$, we can also talk about bimodules (or trimodules, et cetera). These come in several flavors, known as type $AA$, $AD$, $DA$, $DD$. For instance, for differential graded algebras $\AA$ and $\BB$ a left-left type $DD$ bimodule over $(\AA, \BB)$ is simply a left type $D$ module over $\AA \otimes \BB$; the other types are slightly more complicated. The $\AA_\infty$ tensor products of bimodules behave as expected: for instance, given a right type $A$ module $M$ over $\AA$ and a type $DD$ bimodule $N$ over $(\AA, \BB)$, $M \mathbin{\tilde\otimes_\AA} N$ is a type $D$ module over $\BB$. The box tensor product $\boxtimes$ may be used in place of $\tilde\otimes$ under suitable conditions. See \cite[Section 2]{LOTBimodules} for the complete definitions.

\subsection{Invariants of bordered manifolds}

We will focus solely on the case of torus boundary. We consider $T^2 = S^1 \times S^1$, oriented by choosing the same orientation on both $S^1$ factors and taking the product orientation.

The \emph{torus algebra} $\AA = \AA(T^2)$ is freely generated as a vector space over $\F$ by mutually orthogonal idempotents $\iota_0$ and $\iota_1$ and additional elements $\rho_1$, $\rho_2$, $\rho_3$, $\rho_{12}$, $\rho_{23}$, and $\rho_{123}$, with the following nonzero multiplications:
\[
\begin{aligned}
\iota_0 \rho_1 = \rho_1 \iota_1 &= \rho_1 & \iota_1 \rho_2 = \rho_2 \iota_0 &= \rho_2 & \iota_0 \rho_3 = \rho_3 \iota_1 &= \rho_3 \\
\iota_0 \rho_{12} = \rho_{12} \iota_0 &= \rho_{12} & \iota_1 \rho_{23} = \rho_{23} \iota_1 &= \rho_{23} & \iota_0 \rho_{123} = \rho_{123} \iota_1 &= \rho_{123} \\
\rho_1 \rho_2 &= \rho_{12} & \rho_2 \rho_3 &= \rho_{23} & \rho_{12} \rho_3 = \rho_1 \rho_{23} &= \rho_{123}
\end{aligned}
\]
(All other multiplications among the generators zero.) The multiplicative identity in $\AA$ is $\one = \iota_0 + \iota_1$. The differential on $\AA$ is defined to be zero; note that this eliminates the second sum in \eqref{eq:a-relation} and the third term in \eqref{eq:d-relation}.

A \emph{bordered manifold (with torus boundary)} is an oriented $3$-manifold $Y$ along with a diffeomorphism $\phi\co T^2 \to \partial Y$, which we consider up to isotopy fixing a neighborhood of a point. We call $(Y,\phi)$ \emph{type $A$} if $\phi$ is orientation-preserving and \emph{type $D$} if $\phi$ is orientation-reversing. Lipshitz, Ozsv\'ath, and Thurston associate to a type $A$ bordered manifold $(Y_1, \phi_1)$ a type $A$ module $\CFA(Y_1, \phi_1)$ over $\AA$, and to a type $D$ bordered manifold $(Y_2, \phi_2)$ a type $D$ module $\CFD(Y_2, \phi_2)$ over $\AA$. (The maps $\phi_1$ and $\phi_2$ are often suppressed from the notation if they are understood from the context.) Up to the appropriate notion of chain homotopy equivalence, each of these modules is a diffeomorphism invariant of the manifold with parametrized boundary. These invariants are defined in terms of counts of pseudo-holomorphic curves in $\Sigma \times [0,1] \times \R$, where $\Sigma$ is a bordered Heegaard diagram; we shall say nothing more about the definition. The \emph{pairing theorem} states that the Heegaard Floer homology of the closed, oriented $3$-manifold gotten by gluing $Y_1$ to $Y_2$ along their boundaries via the diffeomorphism $\phi_2 \circ \phi_1^{-1}$ is determined by the bordered invariants of $Y_1$ and $Y_2$:
\[
\HF(Y_1 \cup_{\phi_2 \circ \phi_1^{-1}} Y_2) \cong H_*\left( \CFA(Y_1, \phi_1) \mathbin{\tilde\otimes} \CFD(Y_2, \phi_2) \right).
\]
There are also various bimodules associated to manifolds with two boundary components, denoted $\CFAA$, $\CFAD$, $\CFDA$, and $\CFDD$ according to whether the parameterizations of the boundary components are orientation-preserving or orientation-reversing, and similar gluing theorems apply. See \cite{LOTBordered, LOTBimodules} for further details.

Lipshitz, Ozsv\'ath, and Thurston provide a convenient notation for type $D$ modules over the torus algebra $\AA(T^2)$ \cite[Section 11.1]{LOTBordered}. If $I_1, \dots, I_k$ are finite sequences of integers, let $I_1 \cdots I_k$ denote their concatenation. Let $\RR$ denote the set of nonempty, strictly increasing sequences of consecutive integers in $\{1,2,3\}$, and let $\RR' = \RR \cup \{\emptyset\}$. Thus, the non-idempotent generators of $\AA(T^2)$ correspond to elements of $\RR$; for convenience, we define $\rho_\emptyset = \one$.

Let $V = V^0 \oplus V^1$ be a $\Z/2$--graded vector space over $\F=\Z/2\Z$. A \emph{collection of coefficient maps} consists of a linear map $D = D_\emptyset \co V \to V$ taking $V^0$ to $V^0$ and $V^1$ to $V^1$, and, for each $I = (i_1, \dots, i_n) \in \RR$, a linear map $D_I \co V^{[i_1-1]} \to V^{[i_n]}$ (where for $i \in \Z$, $[i] \in \{0,1\}$ denotes the mod-2 reduction of $i$), satisfying the condition that for each $I \in \RR'$,
\begin{equation} \label{eq:coeffmaps}
\sum_{J,K \in \RR' \mid JK = I} D_K \circ D_J = 0,
\end{equation}
where the sum is taken over all pairs of elements in $\RR'$ whose concatenation is $I$. In other words, $D_\emptyset$ is a differential; $D_1$, $D_2$, and $D_3$ are chain maps; $D_{12}$ and $D_{23}$ are nullhomotopies of $D_2 \circ D_1$ and $D_3 \circ D_2$, respectively; and $D_{123}$ is a homotopy between $D_{23} \circ D_1$ and $D_3 \circ D_{12}$. For convenience, we may trivially extend each $D_I$ over all of $V^0 \oplus V^1$. A collection of coefficient maps determines a type $D$ structure on $V$: define multiplication by $\iota_0$ and $\iota_1$ by projection onto $V^0$ and $V^1$ respectively, and for each $v \in V$, define
\[
\delta_1(v) = \sum_{I \in \RR'} \rho_I \otimes D_I(v).
\]
The higher maps $\delta_k$ are then given by compositions of the maps $D_I$:
\[
\delta_k(v) = \sum_{I_1, \dots, I_k \in \RR'} \rho_{I_1} \otimes \dots \otimes \rho_{I_k} \otimes (D_{I_k} \circ \dots \circ D_{I_1})(v).
\]
Furthermore, any type $D$ structure over $\AA$ can be obtained in this manner \cite[Lemma 11.5]{LOTBordered}.

We say that $(V, \{D_I\})$ is \emph{reduced} if $D_{\emptyset}=0$, in which case the relations in \eqref{eq:coeffmaps} simplify to:
\begin{equation} \label{eq:reducedcoeffmaps}
D_2 \circ D_1 = 0 \qquad D_3 \circ D_2 = 0 \qquad D_3 \circ D_{12} = D_{23} \circ D_1.
\end{equation}
It is not hard to see that any type $D$ structure is homotopy equivalent to a reduced one.  See \cite[Section 2.6]{LevineDoublingOperators} for more details.

Finally, if $M$ is a type $A$ module and either $M$ or $V$ is bounded, then the differential on the box tensor product $M \boxtimes V$ is given explicitly by
\[
\partial^\boxtimes(x \otimes y) = \sum_{I_1, \dots, I_r \in \RR} m_{r+1}(x \otimes \rho_{I_1} \otimes \dots \otimes \rho_{I_r}) \otimes (D_{I_r} \circ \dots \circ D_{I_1})(y).
\]
for each $x \in M$ and $y \in V$. (The sum includes an $r=0$ term, where the composition of zero coefficient maps is the identity on $V$.)

\subsection{Computing \texorpdfstring{$\CFA$}{CFA} from \texorpdfstring{$\CFD$}{CFD}}

Let $r \co T^2 \to T^2$ be the orientation-reversing involution that interchanges the two coordinates of $S^1 \times S^1$. This involution gives a one-to-one correspondence between type $A$ and type $D$ bordered manifolds, given by $(Y,\phi) \mapsto (Y, \phi \circ r)$. The bordered invariants of $(Y, \phi)$ and $(Y, \phi \circ r)$ are related by taking tensor products with the appropriate \emph{identity bimodules}, $\CFAA(\I)$ and $\CFDD(\I)$. Here $\I$ denotes the mapping cylinder of the identity map of $T^2$. If $(Y, \phi)$ is a type $A$ bordered $3$-manifold, \cite[Corollary 1.1]{LOTBimodules} says that
\begin{equation} \label{eq:identitybimodule}
\CFD(Y, \phi \circ r) \simeq \CFA(Y,\phi) \mathbin{\tilde\otimes} \CFDD(\I) \quad \text{and} \quad \CFA(Y, \phi) \simeq \CFAA(\I) \mathbin{\tilde\otimes} \CFD(Y, \phi \circ r).
\end{equation}
Here, we view $\CFAA(\I)$ as a right-right $AA$ bimodule and $\CFDD(\I)$ as a left-left $DD$ bimodule, each over $(\AA, \AA)$.\footnote{Our perspective here is slightly different from that of Lipshitz, Ozsv\'ath, and Thurston, who use two distinct algebras associated to $T^2$ and $-T^2$, denoted $\AA(T^2)$ and $\AA(-T^2)$, where $\AA(-T^2) = \AA(T^2)^{\operatorname{op}}$. If $(Y,\phi)$ is a type $A$ bordered manifold (in the sense used above), then one can define $\CFD(Y,\phi)$ as a type $D$ structure over $\AA(-T^2)$, and one views $\CFAA(\I)$ and $\CFDD(\I)$ as $(\AA(T^2), \AA(-T^2))$-bimodules. To see how \eqref{eq:identitybimodule} follows from \cite[Corollary 1.1]{LOTBimodules}, note that the map $r$ (which can be realized as the symmetry of the pointed matched circle associated to the torus) induces an isomorphism between $\AA(T^2)$ and $\AA(-T^2)$, which gives the identification between $\CFD(Y,\phi)$ (as a type $D$ module over $\AA(-T^2)$) and $\CFD(Y, \phi \circ r)$ (as a type $D$ module over $\AA(T^2)$). We find it conceptually simpler to work with a single algebra, at the cost of being more explicit about the role of $r$.} Thus, if a parametrization $\phi$ (either orientation-preserving or orientation-reversing) is understood from context, we will simply speak of $\CFA(Y)$ and $\CFD(Y)$. Note that the two halves of \eqref{eq:identitybimodule} are equivalent statements, since by \cite[Theorem 2.3]{LOTMorphism} the operations of tensoring with $\CFAA(\I)$ and $\CFDD(\I)$ are inverses up to homotopy equivalence. That is, if $M$ is a type $A$ module and $N$ is a type $D$ module, then $M \simeq \CFAA(I) \mathbin{\tilde\otimes} N$ if and only if $M \mathbin{\tilde\otimes} \CFDD(\I) \simeq N$.

We now describe an algorithm for computing $\CFA(Y)$ from $\CFD(Y)$, based on an idea described to us by Peter Ozsv\'ath. The basic idea is as follows. We begin by taking a basis for $\CFD(Y)$ as a basis for $\CFA(Y)$.  The nonzero multiplications on $\CFA(Y)$ are then in bijection with nonzero compositions of the coefficient maps for $\CFD(Y)$. Specifically, if
\[
D_{J_r}\circ D_{J_{r-1}}\circ \cdots \circ D_{J_1}(v) = w,
\]
then
\[
m_{k+1}(v\otimes \rho_{I_1}\otimes \cdots \otimes \rho_{I_k})=w,
\]
where the relationship between $(J_1,\dots,J_r)$ and $(I_1,\dots,I_k)$ is determined by the procedure:
\begin{enumerate}
\item Replace all occurrences of $1$ in the string $J_1\cdots J_r$ with $3$, and vice-versa.
\item Write the resulting string $I$ (uniquely) as a concatenation $I=I_1\cdots I_k$ of increasing sequences $I_i$ satisfying $\last(I_i)>\first(I_{i+1})$ for all $i=1,\dots ,k-1$.
\end{enumerate}
(Here, $\first(I)$ and $\last(I)$ denote the first and last entries of $I$, respectively.) For example, suppose in $\CFD(Y)$ we have $D_{23}\circ D_{23}\circ D_{123}(v)=w$.  We first take the string $1232323$ and replace it with $3212121$, which we then parse as $3,2,12,12,1$.  This tells us that in $\CFA(Y)$ we have a multiplication $m_6(v\otimes \rho_{3}\otimes \rho_{2} \otimes \rho_{12} \otimes \rho_{12} \otimes \rho_{1})=w$. (See Section \ref{sec:examples} for an example of this procedure applied to $\CFD$ of the trefoil complement.)

To be more precise, let $\SS$ denote the set of strictly decreasing, nonempty sequences of consecutive elements of $\{1,2,3\}$,\footnote{We shall write elements of $\RR$ and $\SS$ as strings of digits, without parentheses or commas.} and let $\phi\co \SS \to \RR$ denote the bijection defined by interchanging the roles of $1$ and $3$:
\[
\begin{aligned}
\phi(1) &= 3  & \phi(2) &= 2 & \phi(3) &= 1 \\
\phi(21) &= 23  & \phi(32) &= 12 & \phi(321) &= 123. \\
\end{aligned}
\]
A finite sequence of integers $I$ is called \emph{alternating} if its entries alternate in parity. Strong induction on length shows that for any alternating sequence $I$ of elements of $\{1,2,3\}$, there is a unique decomposition $I = J_1 \cdots J_j$ such that $J_1, \dots, J_j \in \SS$ and for each $i = 1, \dots, j-1$, $\last(J_i) < \first(J_{i+1})$. In this case, we define $\Psi(I) = (J_1, \dots, J_j)$. The following lemma is an immediate consequence of the definition of $\Psi$:

\begin{lemma} \label{lemma:concatenate}
Let $I$ and $I'$ be alternating sequences whose concatenation $II'$ is alternating. Suppose that $\Psi(I) = (J_1, \dots, J_j)$ and $\Psi(I') = (K_1, \dots, K_k)$. Then
\[
\Psi(II') = \begin{cases}
 (J_1, \dots, J_j, K_1, \dots, K_k) & if\ \last(I) < \first(I') \\
 (J_1, \dots, J_{j-1}, J_j K_1, K_2, \dots, K_k) & if\ \last(I) > \first(I').
 \end{cases}
\]
\end{lemma}

The algorithm is given by the following theorem:

\begin{theorem} \label{thm:identity}
Let $(V, \delta_1)$ be a reduced type-$D$ module over $\AA$, seen as a finite-dimensional vector space $V = V^0 \oplus V^1$, with coefficient maps $D_1, D_2, D_3, D_{12}, D_{23}, D_{123}$ satisfying \eqref{eq:reducedcoeffmaps}. For $k \ge 0$, define maps
\[
m_{k+1} \co V \otimes \AA^{\otimes k} \to V
\]
as follows. Set $m_1 = 0$. For $k >1$ and any $I_1, \dots, I_k \in \RR$ whose concatenation $I_1 \cdots I_k$ is alternating and for which $\last(I_i) > \first(I_{i+1})$ for all $i=1, \dots, k-1$, write
\[
\Psi(I_1 \cdots I_k) = (J_1, \dots, J_j)
\]
and define
\[
m_{k+1}(v \otimes \rho_{I_1} \otimes \cdots \otimes \rho_{I_k}) = (D_{\phi(J_j)} \circ \cdots \circ D_{\phi(J_1)})(v)
\]
for all $v \in V$. For any other $I_1, \dots, I_k$, define
\[
m_{k+1}(v \otimes \rho_{I_1} \otimes \cdots \otimes \rho_{I_k})=0.
\]
Then the maps $m_k$ satisfy the $\AA_\infty$ relations. Furthermore, we have
\[
(V, \delta_1) \simeq (V, \{m_k\}) \boxtimes \CFDD(\I) \quad \text{and} \quad (V, \{m_k\}) \simeq \CFAA(\I) \boxtimes (V, \delta_1).
\]
\end{theorem}


\begin{proof}[Proof of Theorem \ref{thm:identity}]
To see that the maps $m_k$ satisfy the $\AA_\infty$ relations, we must show that for any $I_1, \dots, I_k \in \RR$ and any $v \in V$,
\begin{multline} \label{eq:Arelationrho}
\sum_{i=1}^{k-1} m_{k-i+1} ( m_{i+1}(v \otimes \rho_{I_1} \otimes \cdots \otimes \rho_{I_i} ) \otimes \rho_{I_{i+1}} \otimes \cdots \otimes \rho_{I_k} ) + \\
\sum_{i=1}^{k-1} m_{k-1}(v \otimes \rho_{I_1} \otimes \cdots \otimes \rho_{I_{i-1}} \otimes \rho_{I_i} \rho_{I_{i+1}} \otimes \rho_{I_{i+2}} \otimes \cdots \otimes \rho_{I_k}) = 0.
\end{multline}
We may assume that $I_1 \cdots I_k$ is alternating, since otherwise all the terms in \eqref{eq:Arelationrho} would vanish because the tensor products are taken over the ring of idempotents. Indeed, if we had a term such as $\rho_1\otimes \rho_3$ we could write it as $\rho_1\iota_1\otimes\rho_3=\rho_1\otimes \iota_1\rho_3=\rho_1\otimes 0$, with similar expressions for any other non-alternating occurrence.

If $\last(I_i) < \first(I_{i+1})$, then for any $i' \ne i$, the $i'{}\Th$ terms of both sums in \eqref{eq:Arelationrho} must both vanish by definition. Thus, we may assume that there is at most one value of $i$ for which $\last(I_i) < \first(I_{i+1})$. If such an $i$ exists, Lemma \ref{lemma:concatenate} implies that if
\[
\Psi(I_1\cdots I_i)=(L^i_1,\dots, L^i_{\ell_i}), \ \ \Psi(I_{i+1}\cdots I_k)=(M^i_1, \dots, M^i_{m_i}),\  \text{and} \ \Psi(I_1\cdots I_k)=(J_1,\dots, J_j),
\]
then  $\ell_i + m_i = j$ and
\[
(L^i_1, \dots, L^i_{\ell_i}, M^i_1, \dots, M^i_{m_i}) = (J_1, \dots, J_j).
\]
Thus the $i\Th$ term of the first sum in \eqref{eq:Arelationrho} equals
\begin{equation} \label{eq:term}
(D_{\phi(J_j)} \circ \cdots \circ D_{\phi(J_1)})(v).
\end{equation}
Since $\rho_{I_i} \rho_{I_{i+1}} = \rho_{I_i I_{i+1}}$, the $i\Th$ term of the second sum in \eqref{eq:Arelationrho} equals \eqref{eq:term} as well. Thus, the $i\Th$ terms of the two sums in \eqref{eq:Arelationrho} cancel each other, and all other terms in both sums vanish.

Therefore, we may assume that for all $i = 1, \dots, k-1$, we have $\last(I_i) > \first(I_{i+1})$. Since $\rho_{I_i} \rho_{I_{i+1}} = 0$, the entire second sum in \eqref{eq:Arelationrho} vanishes. Suppose that
\[
\Psi(I_1 \cdots I_k) = (J_1, \dots, J_j).
\]
For each $i = 1, \dots, k-1$, suppose that
\[
\Psi(I_1 \dots I_i) = (L^i_1, \dots, L^i_{\ell_i}) \quad \text{and} \quad \Psi(I_{i+1} \dots I_{k}) = (M^i_1, \dots, M^i_{m_i}).
\]
Thus, the first sum in \eqref{eq:Arelationrho} equals
\begin{equation} \label{eq:Arelationrho2}
\sum_{i=1}^{k-1}(D_{\phi(M^i_{m_i})} \circ \cdots \circ D_{\phi(M^i_1)} \circ D_{\phi(L^i_{\ell_i})} \circ \cdots \circ D_{\phi(L^i_1)}) (v).
\end{equation}
Furthermore, Lemma \ref{lemma:concatenate} implies that $\ell_i + m_i - 1 = j$ and
\[
(J_1, \dots, J_j) = (L^i_1, \dots, L^i_{\ell_i-1}, L^i_{\ell_i} M^i_1, M^i_2, \dots, M^i_{m_i}).
\]
In particular, the concatenation $L^i_{\ell_i} M^i_1$ equals either $21$, $32$, or $321$.

If $L^i_{\ell_i} = 3$ and $M^i_1 = 2$, then $D_{\phi(M^i_1)} \circ D_{\phi(L^i_{\ell_i})} = D_2 \circ D_1 = 0$ by \eqref{eq:reducedcoeffmaps}, and therefore the $i\Th$ term of \eqref{eq:Arelationrho2} vanishes. The same argument holds when $L^i_{\ell_i} = 2$ and $M^i_1 = 1$ using the fact that $D_3\circ D_2=0$.

If $L^i_{\ell_i} = 32$ and $M^i_1 = 1$, the concatenation $I_1 \cdots I_i$ ends in $32$, and we must have $i>1$, $\last(I_{i-1})=3$, $I_i = 2$, and $\first(I_{i+1}) =1$. Therefore,
\[
\Psi(I_1 \cdots I_{i-1}) = (L^i_1, \dots, L^i_{\ell_i-1}, 3) \quad \text{and}\quad
\Psi(I_i \cdots I_k) = (21, M^i_2, \dots, M^i_{m_i}).
\]
The sum of the $(i-1)\Th$ and $i\Th$ terms of \eqref{eq:Arelationrho2} then equals
\[
\sum_{i=1}^{k-1}(D_{\phi(M^i_{m_i})} \circ \cdots \circ D_{\phi(M^i_2)} \circ (D_3 \circ D_{12} + D_{23} \circ D_1) \circ D_{\phi(L^i_{\ell_i-1})} \circ \cdots \circ D_{\phi(L^i_1)}) (v),
\]
which vanishes by \eqref{eq:reducedcoeffmaps}.

If $L^i_{\ell_i} = 3$ and $M^{i_1}= 21$, a similar argument shows that the $i\Th$ and $(i+1)\Th$ terms of \eqref{eq:Arelationrho2} cancel. This completes the proof of \eqref{eq:Arelationrho}. Thus, the maps $m_k$ satisfy the $\AA_\infty$ relations.

For the second part of the theorem, as noted in the discussion following \eqref{eq:identitybimodule}, it suffices to show that
\[
(V, \delta_1) \simeq (V, \{m_k\}) \boxtimes \CFDD(\I).
\]
According to \cite[Proposition 10.1]{LOTBimodules}, the left-left $DD$ bimodule $\CFDD(\I)$ has generators $p,q$, with idempotent action given by
\[
(\iota_0 \otimes \iota_0) \cdot p = p \quad \text{and} \quad (\iota_1 \otimes \iota_1) \cdot q = q
\]
and structure map given by
\[
\delta_1 (p) = (\rho_1 \otimes \rho_3 + \rho_3 \otimes \rho_1 + \rho_{123} \otimes \rho_{123}) \otimes q \quad \text{and} \quad \delta_1 (q) = \rho_2 \otimes \rho_2 \otimes p.
\]
Thus, $(V,\{m_k\}) \boxtimes \CFDD(I)$ is isomorphic to $V$ as a vector space. According to the definition of the box tensor product of a type $A$ module and a type $DD$ bimodule, for $v \in V^0$, we have:
\[
\begin{aligned}
\delta_1(v \otimes p) &= \left( \rho_1 \otimes m_2(v, \rho_3) + \rho_3 \otimes m_2(v, \rho_1)  + \rho_{123}\otimes (m_2(v, \rho_{123}) + m_4(v, \rho_3, \rho_2, \rho_1)) \right) \otimes q \\
 & \qquad + \rho_{12} \otimes m_3(v, \rho_3, \rho_2) \otimes p
\\
&= \left(\rho_1 \otimes D_1(v) + \rho_3 \otimes D_3(v) + \rho_{123} \otimes (D_3 \circ D_2 \circ D_1(v) + D_{123}(v)) \right)\otimes q \\
 & \qquad + \rho_{12} \otimes D_{12}(v) \otimes p \\
&= (\rho_1 \otimes D_1(v) + \rho_3 \otimes D_3(v) + \rho_{123} \otimes D_{123}(v)) \otimes q + \rho_{12} \otimes D_{12}(v) \otimes p,
\end{aligned}
\]
where the final line follows from the fact that $D_2 \circ D_1 = D_3 \circ D_2 = 0$. Likewise, for $w \in V^1$,
\[
\begin{aligned}
\delta_1(w \otimes q) &= \rho_2 \otimes m_2(w, \rho_2) \otimes p + \rho_{23} \otimes m_3(w, \rho_2, \rho_1) \otimes q \\
&= \rho_2 \otimes D_2(w) \otimes p + \rho_{23} \otimes D_{23}(w) \otimes q.
\end{aligned}
\]
Thus, the differential on $(V,\{m_k\}) \boxtimes \CFDD(I)$ is equal to the original differential on $V$.
\end{proof}

\subsection{Bordered invariants of knot complements}

If $K$ is a knot in a homology sphere $Y$, and $X_K$ denotes the exterior of $K$, let $\phi_K\co T^2 \to \partial X_K$ be the orientation-reversing parametrization taking $S^1 \times \{\pt\}$ to a $0$-framed longitude of $K$ and $\{\pt\} \times S^1$ to a meridian of $K$. When $Y$ is an L-space, Lipshitz, Ozsv\'ath, and Thurston give a formula for $\CFD(X_K, \phi_K)$ in terms of the knot Floer complex of $(Y,K)$, which we now describe (adding a few details).

We begin by reviewing some facts about knot Floer homology, as defined in \cite{OSzKnot, RasmussenThesis}.  Let $C^- = \CFKm(Y,K)$ denote the knot Floer complex of $K$, a finitely generated free chain complex over $\F[U]$ with a bounded-above filtration
\[
\cdots \subset \FF_i \subset \FF_{i+1} \subset \cdots \subset C^-
\]
such that $U \cdot \FF_i \subset \FF_{i-1}$ for all $i$. The filtered chain homotopy type of this complex  is an invariant of the knot.

For any nonzero $x \in C^-$, let $A(x) = \min\{i \mid x \in \FF_i\}$; we call this the \emph{filtration level} or \emph{Alexander grading} of $x$. Multiplication by $U$ decreases the Alexander grading by one: $A(U \cdot x)=  A(x)-1$.
By convention,  $A(0) = -\infty$. We may assume that $C^-$ is \emph{reduced}, in the sense that for any $x \in C^-$, $\partial x = U \cdot y + z$, where $A(z) < A(x)$; this implies that
\[
\rank_{\F[U]} C^- = \dim_\F \HFK(Y,K) = 2n+1
\]
for some $n \ge 0$. The manner by which knot Floer homology detects the genus \cite[Theorem 1.2]{OSzGenus} implies that
\[
\FF_{g(K)-1} \subsetneq \FF_{g(K)} = C^-, \quad \FF_{-g(K)-1} \subset UC^-, \quad \FF_{-g(K)} \not\subset UC^-.
\]
Let $C^\infty = C^- \otimes_{\F[U]} \F[U,U^{-1}]$, and extend the filtration to $C^\infty$ accordingly. If $\{x_0, \dots, x_{2n}\}$ is a basis for $C^-$ over $\F[U]$, then $\{U^i x_k \mid k=0, \dots, 2n, \ i \in \Z\}$ is a basis for $C^\infty$ over $\F$. We may picture these basis elements living on the integer lattice in $\R^2$, with the element $U^k x_\ell$ at the point $(-k, A(x_\ell)-k)$. We refer to the coordinates in the plane as $i$ and $j$, and we identify $C^-$ with the subcomplex $C\{i \le 0\} \subset C^\infty$ generated by the basis elements at lattice points with $i \le 0$. The complexes $C\{i \le s\}$ ($s \in \Z$) provide a second filtration on $C^-$ and $C^\infty$.

Let $C^v = C^- / UC^-$ and $C^h = \FF_0(C^\infty) / \FF_{-1}(C^\infty)$, and let $\partial^v$ and $\partial^h$ denote respective induced differentials. We refer to $(C^v, \partial^v)$ and $(C^h, \partial^h)$ as the {\em vertical} and {\em horizontal complexes}, respectively.

The associated graded object of $C^-$ (with respect to the original filtration) is the free $\F[U]$-module
\[
\gr(C^-) = \bigoplus_{i \in \Z} \FF_i / \FF_{i-1},
\]
with the induced multiplication by $U$. Note that the direct sum is as an $\F$-vector space, and not as an $\F[U]$-module since multiplication by $U$ decreases the filtration by one.  For $x \in C^-$, let $[x] \in \gr(C^-)$ denote the image of $x$ in $\FF_{A(x)}/\FF_{A(x)-1}$. Note that $[Ux] = U[x]$. A basis $\{x_0, \dots, x_{2n}\}$ for $C^-$ is called a \emph{filtered basis} if $\{[x_0], \dots, [x_{2n}]\}$ is a basis for $\gr(C^-)$ over $\F[U]$. Any two filtered bases $\{x_0, \dots, x_{2n}\}$ and $\{x'_0, \dots, x'_{2n}\}$ are related by a \emph{filtered change of basis}: if $x_i = \sum_j a_{ij} x'_j$ and $x'_i = \sum_j b_{ij} x_j$, where $a_{ij}, b_{ij} \in \F[U]$, then
\[
A(a_{ij} x'_j) \le A(x_i) \quad \text{and} \quad  A(b_{ij} x_j) \le A(x'_i) \quad \text{for all } i,j.
\]
In particular, if $a_{ij} \not\equiv 0 \pmod {UC^-}$, then $A(x'_j) \le A(x_i)$, and similarly for $b_{ij}$.

A key tool for our main theorem is a formula which expresses $\CFD(X_K)$ in terms of $\CFKm(Y,K)$.  The most useful way to express this formula is by picking a basis for $\CFKm(Y,K)$ and describing $\CFD(X_K)$ in terms of this basis.  To do this it will be useful to have particularly nice bases for $\CFKm$, whose definitions we now recall.

\begin{definition}\label{def:vertbasis}
A filtered basis $\{\xi_0, \dots, \xi_{2n}\}$ for $C^-$ over $\F[U]$ is called {\em vertically simplified} if, for  $j=1, \dots, n$,
\[
A(\xi_{2j-1}) - A(\xi_{2j}) = k_j > 0\quad \text{and}\quad \partial \xi_{2j-1} \cong \xi_{2j} \pmod {UC^-},
\]
while for $p = 0, 1, \dots, n$, we have
\[
\partial \xi_{2p} \in UC^-.
\]
We say that there is a \emph{vertical arrow of length $k_j$ from $\xi_{2j-1}$ to $\xi_{2j}$} and that $\xi_0$ is the \emph{generator of vertical homology}.
\end{definition}

The name is motivated by the fact that in a vertically simplified basis the differential on the vertical complex $(C^v, \partial^v)$ is particularly simple; indeed, in a vertically simplified basis $\partial^v$ can be represented by a collection of vertical arrows which pair up the even and odd basis elements, and where $\xi_0$ has no incoming or outgoing arrows.  Similarly, for the horizontal complex we have

\begin{definition}\label{def:horbasis}
A filtered basis $\{\eta_0, \dots, \eta_{2n}\}$ for $C^-$ over $\F[U]$ is called {\em horizontally simplified} if, for $j = 1, \dots, n$,
\[
A(\eta_{2j}) - A(\eta_{2j-1}) = \ell_j > 0\quad  \text{and}\quad \partial \eta_{2j-1} \equiv U^{\ell_j} \eta_{2j} \pmod {\FF_{A(\eta_{2j-1})-1}},
\]
while for $p = 0, 1, \dots, n$, we have
\[
A(\partial \eta_{2p}) < A( \eta_{2p}).
\]
We say that there is a \emph{horizontal arrow of length $\ell_j$ from $\eta_{2j-1}$ to $ \eta_{2j}$} and that $\eta_0$ is the \emph{generator of horizontal homology}.
\end{definition}

Lipshitz, Ozsv\'ath, and Thurston showed that $C^-$ always admits both horizontally and vertically simplified bases \cite[Proposition 11.52]{LOTBordered}. Furthermore, for any vertically simplified basis $\{\xi_0, \dots, \xi_{2n}\}$ and horizontally simplified basis $\{\eta_0, \dots, \eta_{2n}\}$, the unordered tuples $\{k_1, \dots, k_n\}$ and $\{\ell_1, \dots, \ell_n\}$ are equal; this follows from the symmetry of knot Floer homology under reversing the knot orientation \cite[Section 3.5]{OSzKnot}.

Two particularly useful derivatives of the filtered chain homotopy type of $C^\infty$ can be expressed easily in terms of a vertical or horizontally simplified basis.  The first is the Ozsv\'ath-Szab\'o  concordance invariant \cite{OSz4Genus, RasmussenThesis}.  Denoted $\tau(K)$, this invariant is a homomorphism from the smooth concordance group to the integers which bounds the smooth $4$-genus:
\[
|\tau(K)| \le g_4(K).
\]
In terms of a vertically simplified basis, we have
\[
\tau(K)=A(\xi_0),
\]
while in terms a horizontally simplified basis we have
\[
\tau(K)=-A(\eta_0).
\]
The latter equality again follows from the orientation reversal symmetry.

The second invariant we derive from $C^\infty$ is Hom's invariant $\epsilon(K)\in \{-1,0,1\}$, which captures whether the four-dimensional cobordisms obtained by attaching two-handles to $Y$ along $K$ induce nontrivial maps on Floer homology in certain Spin$^c$-structures \cite[Definition 3.1]{HomComplex}. This invariant can also be expressed in terms of vertically and horizontally simplified bases. Let $[\eta_0]$ denote the image of $\eta_0$ in the vertical complex $C^v$. Also, viewing $\xi_0$ as an element of $C^\infty$, the chain $\xi_0' = U^{A(\xi_0)} \xi_0$ is in $\FF_0$, so we may consider its image $[\xi_0']$ in the horizontal complex $C^h$. Then:
\begin{itemize}
\item If $\epsilon(K) = -1$, then $\partial^v[\eta_0] \ne 0$ and $\partial^h[\xi_0'] \ne 0$.

\item If $\epsilon(K) = 0$, then $[\eta_0] \in \ker\partial^v \minus \im\partial^v$ and $[\xi_0'] \in \ker \partial^h \smallsetminus \im \partial^h$.

\item If $\epsilon(K) = 1$, then $[\eta_0] \in \im\partial^v$ and $[\xi_0'] \in \im\partial^h$.
\end{itemize}

%

The following proposition tells us that the change of basis passing between horizontally and vertically simplified bases can be assumed to be relatively well-behaved.

\begin{proposition} \label{prop:bases}
There exist filtered bases $\{\xi_0, \dots, \xi_{2n}\}$ and $\{\eta_0, \dots, \eta_{2n}\}$ for $C^-$ over $\F[U]$ with the following properties:
\begin{enumerate}
\item \label{item:vertsimp}
$\{\xi_0, \dots, \xi_{2n}\}$ is a vertically simplified basis.

\item \label{item:horizsimp}
$\{\eta_0, \dots, \eta_{2n}\}$ is a horizontally simplified basis.

%

\item \label{item:epsilon}
If $\epsilon(K) = -1$, then $\xi_0 = \eta_1$. If $\epsilon(K) = 0$, then $\xi_0 = \eta_0$. If $\epsilon(K) = 1$, then $\xi_0 = \eta_2$.

\item \label{item:changeofbasis}
If
\[
\xi_p = \sum_{q=0}^{2n} a_{p,q} \eta_q \quad \text{and} \quad \eta_p = \sum_{q=0}^{2n} b_{p,q} \xi_q,
\]
where $a_{p,q}, b_{p,q} \in \F[U]$, then
\[
a_{p,q}=0 \quad \text{whenever } A(\xi_p) \ne A(a_{p,q} \eta_q),
\]
and
\[
b_{p,q} = 0 \quad \text{whenever } A(\eta_p) \ne A(b_{p,q} \xi_q).
\]
In other words, each $\xi_p$ is an $\F[U]$-linear combination of the elements $\eta_q$ that are the same filtration level as $\xi_p$, and vice versa.
\end{enumerate}
\end{proposition}

\begin{proof}
According to Hom \cite[Lemmas 3.2, 3.3]{HomTau}, we may find vertically and horizontally simplified bases $\{\xi_0, \dots, \xi_{2n}\}$ and $\{\eta'_0, \dots, \eta'_{2n}\}$ satisfying \eqref{item:epsilon} (with $\eta_i$ replaced by $\eta'_i$). We shall modify the latter basis to produce a new basis $\{\eta_0, \dots, \eta_{2n}\}$ satisfying the conclusions of the proposition.

As above, any two filtered bases are related by a filtered change of bases, so let
\[
\xi_p = \sum_{q=0}^{2n} a'_{p,q} \eta'_q \quad \text{and} \quad \eta'_p = \sum_{q=0}^{2n} b'_{p,q} \xi_q.
\]
be filtered change of bases. That is, for all $p, q \in \{0, \dots, 2n\}$, we have
\[
A(a'_{p,q} \eta'_q) \le A(\xi_p)  \quad \text{and} \quad A(b'_{p,q} \xi_q) \le A(\eta'_p).
\]
Let
\[
b_{p,q} = \begin{cases} b'_{p,q} &  \text{if}\ A(b'_{p,q} \xi_q) = A(\eta'_p)  \\
0 &  \text{if}\ A(b_{p,q} \xi_q) < A(\eta'_p), \end{cases}
\]
and define
\[
\eta_p = \sum_{q=0}^{2n} b_{p,q} \xi_q \quad \text{and} \quad \Delta_p = \eta'_p - \eta_p.
\]
Note that $A(\eta_p) = A(\eta'_p)$, while $A(\Delta_p) < A( \eta'_p)$. The change-of-basis matrix $(b_{p,q})$ is in block-diagonal form (after reordering rows and columns according to filtration level), so its inverse is as well. Thus the bases $\{\xi_0, \dots, \xi_{2n}\}$ and $\{\eta_0, \dots, \eta_{2n}\}$ satisfy \eqref{item:changeofbasis}. Furthermore, if $i \in \{0,1,2\}$ is the index for which $\xi_0 = \eta_i'$, then $\eta_i = \eta_i'$ by construction, so \eqref{item:epsilon} also holds.

It remains to show that the basis $\{\eta_0, \dots, \eta_{2n}\}$ is horizontally simplified. We have, for any $j = 1, \dots, n$,
\[
\begin{aligned}
\partial \eta_{2j-1} &= \partial \eta'_{2j-1} - \partial \Delta_{2j-1} \\
&\equiv \partial \eta'_{2j-1} \pmod {\FF_{A(\eta'_{2j-1})-1}} \\
&\equiv U^{\ell_j} \eta'_{2j} \pmod {\FF_{A(\eta'_{2j-1})-1}} \\
&= U^{\ell_j} \eta_{2j} + U^{\ell_j} \Delta_{2j} \\
&\equiv U^{\ell_j} \eta_{2j} \pmod {\FF_{A(\eta_{2j-1})-1}},
\end{aligned}
\]
where the last line follows from the fact that
\[
A(U^{\ell_j} \Delta_{2j}) = A(\Delta_{2j}) - \ell_j < A(\eta'_{2j}) - \ell_j = A(\eta'_{2j-1}) = A(\eta_{2j-1}).
\]
Likewise, for $j = 0, 1, \dots, n$, we have
\[
\partial \eta_{2j} = \partial \eta'_{2j} + \partial \Delta_{2j} \in \FF_{A(\eta'_{2j})-1},
\]
as required.
\end{proof}

For the remainder of this section, choose vertically and horizontally simplified bases $\{\tilde \xi_0, \dots, \tilde \xi_{2n}\}$ and $\{\tilde \eta_0, \dots, \tilde \eta_{2n}\}$ for $\CFKm(Y,K)$ satisfying the conclusions of Proposition \ref{prop:bases} above.\footnote{We use tildes for the generators of $\CFKm(Y,K)$ in order to distinguish them from the corresponding elements of $\CFD(X_K)$.} Assume that
\[
\tilde\xi_p = \sum_{q=0}^{2n} \tilde a_{p,q} \tilde\eta_q \quad \text{and} \quad \tilde \eta_p = \sum_{q=0}^{2n} \tilde b_{p,q} \tilde \xi_q,
\]
where $\tilde a_{p,q}, \tilde b_{p,q} \in \F[U]$, and let
\[
a_{p,q} = \tilde a_{p,q} |_{U=0} \quad \text{and} \quad b_{p,q} = \tilde b_{p,q} |_{U=0}.
\]
According to \cite[Theorem 11.27 and Theorem A.11]{LOTBordered}, $\CFD(X_K, \phi_K)$ is completely determined by the lengths of the arrows (i.e., $k_j$ and $\ell_j$), $\tau(K)$, and the change-of-basis matrix $(a_{p,q})$, as follows.

\begin{theorem} \label{thm:cfkcfd}
With notation as above, $\CFD(X_K)$ satisfies the following properties:

\begin{itemize}
\item The summand $\iota_0 \CFD(X_K)$ has dimension $2n+1$, with designated bases $\{\xi_0, \dots, \xi_{2n}\}$ and $\{\eta_0, \dots, \eta_{2n}\}$ related by
\[
\xi_p = \sum_{q=0}^{2n} a_{p,q} \eta_q \quad \text{and} \quad \eta_p = \sum_{q=0}^{2n} b_{p,q} \xi_q.
\]

\item The summand $\iota_1 \CFD(X_K)$ has dimension $\sum_{j=1}^n (k_j + l_j) + t$, where $t = 2 \abs{\tau(K)}$, with basis
\[
\bigcup_{j=1}^n \{\kappa^j_1, \dots, \kappa^j_{k_j}\} \cup \bigcup_{j=1}^n \{\lambda^j_1, \dots, \lambda^j_{l_j}\} \cup \{\mu_1, \dots, \mu_t\} .
\]

\item For $j=1,\dots, n$, corresponding to the vertical arrow $\tilde\xi_{2j-1} \to \tilde\xi_{2j}$ of length $k_j$, there are coefficient maps
\begin{equation} \label{eq:vertchain}
\xi_{2j} \xrightarrow{D_{123}} \kappa^j_1 \xrightarrow{D_{23}} \cdots \xrightarrow{D_{23}} \kappa^j_{k_j} \xleftarrow{D_1} \xi_{2j-1}.
\end{equation}

\item For $j=1, \dots, n$, corresponding to the horizontal arrow $\tilde\eta_{2j-1} \to \tilde\eta_{2j}$ of length $l_j$, there are coefficient maps
\begin{equation} \label{eq:horizchain}
\eta_{2j-1} \xrightarrow{D_3} \lambda^j_1 \xrightarrow{D_{23}} \cdots \xrightarrow{D_{23}} \lambda^j_{l_j} \xrightarrow{D_2} \eta_{2j},
\end{equation}

\item Depending on $\tau(K)$, there are additional coefficient maps
\begin{equation} \label{eq:unstchain}
\begin{cases}
\eta_0 \xrightarrow{D_3} \mu_1 \xrightarrow{D_{23}} \cdots \xrightarrow{D_{23}} \mu_t \xleftarrow{D_1} \xi_0 & \tau(K)>0  \\
\xi_0 \xrightarrow{D_{12}} \eta_0 & \tau(K) =0 \\
\xi_0 \xrightarrow{D_{123}} \mu_1 \xrightarrow{D_{23}} \cdots \xrightarrow{D_{23}} \mu_t \xrightarrow{D_2} \eta_0 & \tau(K) < 0.
\end{cases}
\end{equation}
\end{itemize}
We refer to the subspaces of $\CFD(X_K)$ spanned by the generators in \eqref{eq:vertchain}, \eqref{eq:horizchain}, and \eqref{eq:unstchain} as the \emph{vertical chains}, \emph{horizontal chains}, and \emph{unstable chain}, respectively.\footnote{Note that our notation differs slightly from that of \cite{LOTBordered}: the generators $\kappa^j_1, \dots, \kappa^j_{k_j}$ are indexed in the reverse order, as are $\mu_1, \dots, \mu_t$ in the case where $\tau(K)>0$.}
\end{theorem}


As described in \cite[Lemma 11.40]{LOTBordered}, $\CFD(X_K)$ admits a grading by half-integers, taking integer values on $\iota_0 \CFD(X_K)$ and non-integer values on $\iota_1 \CFD(X_K)$:
\begin{equation} \label{eq:sgrading}
\iota_0 \CFD(X_K) = \bigoplus_{s \in \Z} \CFD(X_K, s) \quad \text{and} \quad \iota_1 \CFD(X_K) = \bigoplus_{s \in \Z+\frac12} \CFD(X_K,s).
\end{equation}
We refer to this grading as the \emph{Alexander grading}. This is justified since there are canonical identifications:
\[
\iota_0 \CFD(X_K) \cong \HFK(Y,K) \quad \text{and} \quad \iota_1 \CFD(X_K) \cong \HFL(K),
\]
where the latter invariant is the longitude Floer homology \cite{EftekharyWhitehead}. Under these identifications the grading on the summands of $\CFD(X_K)$ agrees with the Alexander gradings on each of these groups. Proposition \ref{prop:bases} implies that the Alexander gradings (in $\CFD(X_K)$) of $\xi_0, \dots, \xi_{2n}$ and $\eta_0, \dots, \eta_{2n}$ are equal to the filtration levels (in $\CFKm(Y,K)$) of $\tilde\xi_0, \dots, \tilde\xi_{2n}$ and $\tilde\eta_0, \dots, \tilde\eta_{2n}$, respectively, and that the change of basis is homogeneous. We denote the grading of a homogeneous element $x$ by $A(x)$, and for each $s \in \frac12 \Z$, let $\pi_s\co \CFD(X_K) \to \CFD(X_K, s)$ be the projection map coming from \eqref{eq:sgrading}. The Alexander grading on $\CFD(X_K)$ will eventually enable us to isolate certain pieces of the chain complex for a spliced manifold. As seen in \cite{LOTBordered}, the coefficient maps on $\CFD(X_K)$ are all homogeneous with respect to $A$, with the following degrees:

\begin{center}  \begin{tabular} {| c || c | c | c | c | c | c |}
   \hline
    \text{Coefficient map} & $D_1$ & $D_2$ & $D_3$ & $D_{12}$ & $D_{23}$ & $D_{123}$ \\ \hline
    $A$--degree & $-\frac12$ & $\frac12$ &$ \frac12$ & $0$ &$1$ &$ \frac12 $\\ \hline
  \end{tabular}
\end{center}


Note that $\CFD(X_K)$ need not be bounded; for instance, if $K$ is the unknot, then $\CFD(X_K)$ has a single generator $\xi$, with $\delta_1(\xi) = \rho_{12} \otimes \xi$, and therefore $\delta_k(\xi) = \rho_{12} \otimes \dots \otimes \rho_{12} \otimes \xi$ for all $k \ge 1$. To avoid this issue, we may introduce a modified version of $\CFK(X_K)$ when $\tau(K)=0$.\footnote{We are grateful to Jonathan Hanselman for pointing out this argument, which also appears in his forthcoming preprint.} Let $\CFD(X_K)'$ be just as in Theorem \ref{thm:cfkcfd}, except that the unstable chain in the $\tau(K)=0$ case is replaced with
\[
\xi_0 \xrightarrow{D_1} \nu_1 \xleftarrow{D_\emptyset} \nu_2 \xrightarrow{D_2} \eta_0.
\]
The generators $\nu_1$ and $\nu_2$ have Alexander grading $-\frac12$. It is easy to verify that $\CFD(X_K')$ is chain homotopy equivalent as a type-$D$ module to $\CFD(X_K)$.

\begin{lemma} \label{lemma:bounded}
For any knot $K \subset S^3$, $\CFD(X_K)'$ is bounded.
\end{lemma}

\begin{proof}
We must show that every sufficiently long composition of the coefficient maps on $\CFD(X_K)$ is zero. Note that the vertical chains are closed with respect to any of the coefficient maps: for any $I \in \RR$, $D_I ( \Span(\kappa^j_1, \dots, \kappa^j_{k_j}) \subset \Span(\kappa^j_1, \dots, \kappa^j_{k_j})$. Also, the only compositions of coefficient maps whose restrictions to $\Span(\kappa^j_1, \dots, \kappa^j_{k_j})$ are nonzero are $D_{23}^m$ for $m <k_j$. An analogous statement is true for the unstable chain when $\tau(K)\ge 0$. Therefore, for the purpose of boundedness, it suffices to consider only sequences of coefficient maps made from the horizontal chains and from the unstable chain when $\tau(K) < 0$. Since these maps do not involve $D_1$ or $D_{12}$,  they all increase the Alexander grading by $\frac12$ or $1$. Since the Alexander gradings of all elements are between $-g(K)$ and $g(K)$, this implies that any composition of more than $4g$ coefficient maps is equal to zero.
\end{proof}


\section{Splicing knot complements}


For any knots $K_1 \subset Y_1$ and $K_2 \subset Y_2$, note that the composition
\[
\phi_{K_2} \circ r \circ \phi_{K_1}^{-1} \co \partial X_{K_1} \to \partial X_{K_2}
\]
is orientation-reversing and takes a $0$-framed longitude of $K_1$ to a meridian of $K_2$ and a meridian of $K_1$ to a $0$-framed longitude of $K_2$. The manifold gotten by gluing $X_{K_1}$ and $X_{K_2}$ via $\phi_{K_2} \circ r \circ \phi_{K_1}^{-1}$ is thus precisely $Y(K_1,K_2)$, as defined in the introduction. Therefore, we have
\[
\HF(Y(K_1,K_2)) = H_*(\CFA(X_{K_1}, \phi_{K_1} \circ r) \mathbin{\tilde\otimes} \CFD(X_{K_2}, \phi_{K_2})).
\]
(Henceforth, we suppress the parametrizations from the notation.) Since $\CFD(X_{K_2}) \simeq \CFD(X_{K_2})'$ and the latter is bounded, there are homotopy equivalences
\begin{align*}
\CFA(X_{K_1}) \mathbin{\tilde\otimes} \CFD(X_{K_2}) &\simeq \CFA(X_{K_1}) \mathbin{\tilde\otimes} \CFD(X_{K_2})' \\
&\simeq \CFA(X_{K_1}) \boxtimes \CFD(X_{K_2})'.
\end{align*}
We obtain $\CFA(X_{K_1})$ by applying Theorem \ref{thm:identity} to $\CFD(X_{K_1})$; we do not need to use the bounded version.

Our strategy will be to describe, for any knot $K$ in an L-space homology sphere $Y$, the behavior of those elements of $\CFD(X_K)$ and $\CFD(X_K)'$ that come from the part of knot Floer homology in lowest Alexander grading, $\HFK(Y,K, -g(K))$. We will then use Theorem \ref{thm:identity} to describe the corresponding elements of $\CFA(X_K)$. The tensor products of these elements will give rise to the homology classes in
\[
H_*(\CFA(X_{K_1}) \boxtimes \CFD(X_{K_2})')
\]
needed for Theorem \ref{thm:main}.

Assume, for the duration of this section, that we have bases $\{\tilde \xi_0,\dots, \tilde \xi_{2n}\}$ and $\{\tilde \eta_0, \dots, \tilde \eta_{2n}\}$ for $\CFKm(Y,K)$ just as in Theorem \ref{thm:cfkcfd}. We begin by considering the basis elements that have Alexander grading equal to $-g(K)$.

By the definitions of vertically and horizontally reduced bases, for any $j \in \{1, \dots, n\}$, neither $\tilde\xi_{2j-1}$ nor $\tilde\eta_{2j}$ can have Alexander grading equal to $-g(K)$, since that would require $\tilde \xi_{2j}$ or $\tilde \eta_{2j-1}$ to have Alexander grading less than $-g(K)$. Furthermore, if $K$ is a nontrivial knot and $A(\tilde \xi_0) = -g(K)$, then $\tau(K) = -g(K)$ and $\epsilon(K) = -1$, since $\tilde\xi_0$ is congruent modulo $UC^-$ to a linear combination of $\tilde \eta_1, \dots, \tilde \eta_{2j-1}$. Likewise, if $A(\tilde \eta_0) = -g(K)<0$, then $\tau(K) = g(K)$ and $\epsilon(K) = 1$.

\begin{lemma} \label{lemma:length1}
If $A(\tilde\eta_{2j-1}) = -g(K)$ and $\ell_j = 1$, then $\tilde\eta_{2j}$ is congruent modulo $UC^-$ to a linear combination of $\tilde\xi_0,\tilde\xi_2, \tilde\xi_4, \dots, \tilde\xi_{2n}$. Furthermore, the coefficient of $\tilde\xi_0$ is zero unless  $\tau(K) = -g(K)+1$ and $g(K)>1$.
\end{lemma}

\begin{proof}
For $p = 0, \dots, 2n$, let $\hat \xi_p$ and $\hat \eta_p$ respectively denote the images of $\tilde \xi_p$ and $\tilde \eta_p$ in $C^v = C^-/UC^-$. The elements $\hat\xi_0, \hat\xi_2, \dots, \hat\xi_{2n}$ generate the cycles for $\partial^v$, so we must show that $\partial^v \hat\eta_{2j} = 0$, i.e., that $\partial \tilde\eta_{2j} \in UC^-$. When $\ell_j=1$, the definition of a horizontally simplified basis says that $A(\tilde\eta_{2j}) = A(\tilde\eta_{2j-1})+1$ and that $\partial  \tilde\eta_{2j-1} = U \tilde\eta_{2j} + \epsilon$, where $A(\epsilon) < -g(K)$. The fact that $\hat\eta_{2j-1}$ is in the minimal Alexander grading on $C^-/ UC^-$ implies that $\epsilon = U \delta$ for some $\delta$ with $A(\delta) \le -g(K)$. We have:
\[
0 = \partial^2 \tilde\eta_{2j-1}  = U \partial \tilde\eta_{2j} + \partial \epsilon = U(\partial \tilde\eta_{2j} + \partial \delta),
\]
and since multiplication by $U$ is injective, $\partial \tilde\eta_{2j} = \partial \delta$. Now, since $C^-$ is reduced, we have $\partial \delta = U \alpha + \beta$, where $A(\beta) < A(\delta) \le A(\tilde\eta_{2j-1})$, and therefore $\beta = U\gamma$ as above. Thus $\partial \tilde\eta_{2j} = U(\alpha + \gamma)$, as required.

Furthermore, if $\hat\eta_{2j}$ has a $\hat\xi_0$ component, then the Alexander grading of $\hat\xi_0$ --- which by definition is $\tau(K)$ --- is equal to $-g(K)+1$, and $\epsilon(K) = 1$ since $\hat\xi_0$ has an incoming horizontal arrow. The fact that $g(K)>1$ then follows from Lemma \ref{lemma:g1tau0} below.
\end{proof}

\begin{lemma} \label{lemma:g1tau0}
If $Y$ is an L-space homology sphere, and $K \subset Y$ is a knot with $g(K)=1$ and $\tau(K) = 0$, then $\epsilon(K)=0$.
\end{lemma}

\begin{proof}
Suppose, toward a contradiction, that $\epsilon(K)=1$. We may find horizontally and reduced bases $\{\eta_0, \dots, \eta_{2n}\}$ and $\{\xi_0, \dots, \xi_{2n}\}$ satisfying the conclusions of Proposition \ref{prop:bases}; in particular, $\xi_0 = \eta_2$. Since $g(K) = 1$, the horizontal arrow from $\eta_1$ to $\eta_2$ has length $1$, which means that $A(\eta_1) = -1$ and $\partial \eta_1 = U \xi_0 + \gamma$, where $A(\gamma) < -1$. As above, we have $\gamma = U\delta$ for some $\delta$ with $A(\delta) \le -1$ since there are no chains with $U$ power zero having Alexander grading less than $-g=-1$. Now the filtration levels of $\xi_0$ and each of the $\xi_{2j-1}$ are strictly greater than $-1$, because the vertical differential decreases the Alexander grading and $A(\xi_0)=\tau(K)=0$.  It follows that $\eta_1$ is in the span of $\{\xi_2, \dots, \xi_{2n}\}$, so there exist elements $\xi, \alpha$ such that $\partial \xi = \eta_1 + U \alpha$. Hence,
\[
0 = \partial^2 \xi = \partial \eta_1 + U \partial \alpha = U(\xi_0 + \delta + \partial \alpha),
\]
so, by the injectivity of multiplication by $U$
\[
\partial \alpha = \xi_0 + \delta.
\]
If we write $\delta = a_0 \xi_0 + \dots + a_{2n} \xi_{2n}$, where $a_i \in \F[U]$, the fact that $A(\delta) \le -1$ implies that $a_0$ and $a_1, a_3, \dots, a_{2n-1}$ must be divisible by $U$. Setting
\[
\alpha' = \alpha + \sum_{j=1}^n a_{2j} \xi_{2j-1},
\]
we see that
\[
\partial \alpha' \equiv \xi_0 \pmod {UC^-},
\]
which means that $\xi_0$ is in the image of the vertical differential, a contradiction.

If $\epsilon(K) = -1$, we reduce to the previous case by considering the mirror $\overline{K}$ in place of $K$.
\end{proof}

We now return to the bordered invariants. For a nontrivial knot $K$, let
\[
B_K = \CFD(X_K, -g(K)).
\]
Note that
\[
B_K \cong \HFK(Y,K,-g(K)),
\]
and it is generated by some subset of $\{\xi_{2j} \mid j=1, \dots, n\}$, along with $\xi_0$ if $\tau(K) = -g(K)$; it is also generated by some subset of $\{\eta_{2j-1} \mid j = 1, \dots, n\}$, along with $\eta_0$ if $\tau(K) = g(K)$. Let $\pi_B = \pi_{-g}$ denote the projection onto $B_K$.

Note that $\CFD(X_K, -g(K) + \frac12)$ is generated by the elements $\kappa^j_1$, $\lambda^j_1$, and $\mu_1$ that are ``adjacent'' to the generators of $\CFD(X_K, -g(K))$ in the vertical, horizontal, and unstable chains. To be precise, let
\[
V_K= \text{subspace generated by }\{\kappa^j_1 \mid A(\tilde\xi_{2j}) = -g(K)\},\ \text{and}\  \mu_1\ \text{if}\ \tau(K) = -g(K),
\]
\[
H_K= \text{subspace generated by } \{\lambda^j_1 \mid A(\tilde\eta_{2j-1}) = -g(K)\}, \ \text{and}\ \mu_1\ \text{if}\ \tau(K) = g(K).
\]
Clearly, $\CFD(X_K, -g(K)+\frac12) = V_K \oplus H_K$. Furthermore, $V_K$ and $H_K$ each have the same rank as $B_K$; indeed, the restrictions of $D_{123}$ and $D_{3}$ to $B_K$ gives isomorphisms from $B_K$ to $V_K$ and $H_K$, respectively. Let $\pi_V\co \CFD(X_K) \to V_K$ and $\pi_H\co \CFD(X_K) \to H_K$ be the composition of $\pi_{-g+1/2}$ with projection onto the appropriate factors.

In $\CFD(X_K)'$, let $B_K'$, $V_K'$, and $H_K'$ be defined identically, except that $H_K$ also includes $\nu_1$ and $\nu_2$ in the case where $g(K)=1$ and $\tau(K)=0$.

The next two propositions describe all of the differentials into and out of $B_K$ and $V_K$, as well as their  counterparts in $\CFD(X_K)'$. ($H_K$ turns out not to be as useful for the present purposes.)

\begin{proposition} \label{prop:BK-D}
Let $K$ be a nontrivial knot with genus $g>0$ in an L-space homology sphere, and consider the subspace $B_K \subset \CFD(X_K)$ as described above.
\begin{enumerate}
\item \label{item:BKin}
Elements of $B_K$ have no incoming coefficient maps of any type. More precisely, for each $I \in \RR$, we have $\pi_{B} \circ D_I = 0$.

\item \label{item:BKout}
If $I_1, \dots, I_r$ are elements of $\RR$ such that the restriction of $D_{I_r} \circ \dots \circ D_{I_1}$ to $B_K$ is nontrivial, then:
\begin{enumerate}
    \item $I_1 = 3$ or $123$.
    \item If $I_1 = 123$ and $r>1$, then $I_2 = 23$.
    \item If $I_1 = 3$ and $r>1$, then $I_2 = 2$ or $23$; if $I_2 = 2$ and $r>2$, then $I_3 = 123$.
\end{enumerate}
\end{enumerate}
The same is true for $B_K' \subset \CFD(X_K)'$.
\end{proposition}

\begin{proof}
The proofs for $B_K$ and $B_K'$ are identical, so we consider only $B_K$. The first statement follows immediately from Theorem \ref{thm:cfkcfd} and the fact that $B_K$ does not contain elements of the form  $\eta_{2j}$ for $j = 1, \dots, n$, and does not contain $\eta_0$ if $\tau(K)\le 0$ (the only cases where $\eta_0$ has an incoming coefficient map).

For the second statement, note that $D_1$ and $D_{12}$ restricted to $B_K$ are both zero, so we may reduce to the two cases where $I_1 = 3$ or $123$, which we treat separately.

In the case where $I_1 = 123$, we consider the vertical basis for $B_K$. If $\xi_{2j} \in B_K$, where $j \in \{1, \dots, n\}$, then the only nonzero sequence of coefficient maps coming from $\xi_{2j}$ and starting with $D_{123}$ is the vertical chain
\[
\xi_{2j} \xrightarrow{D_{123}} \kappa^j_1 \xrightarrow{D_{23}} \cdots \xrightarrow{D_{23}} \kappa^j_{k_j}.
\]
If $\xi_0 \in \CFD(X_K, -g)$, then $\tau(K) = -g(K) < 0$, so the unstable chain provides the sequence
\[
\xi_0 \xrightarrow{D_{123}} \mu_1 \xrightarrow{D_{23}} \dots \xrightarrow{D_{23}} \mu_{2g} \xrightarrow{D_2} \eta_0,
\]
with at least one $D_{23}$. Thus, the only $I$ such that $D_I \circ D_{123} |_{B_K}$ can be nonzero is $I=23$.

In the case where $I_1 = 3$, we use the horizontal basis. If $\eta_0 \in B_K$, then $\tau(K) = g(K)>0$, so the unstable chain provides the sequence
\[
\eta_0 \xrightarrow{D_{3}} \mu_1 \xrightarrow{D_{23}} \cdots \xrightarrow{D_{23}} \mu_{2g}.
\]
If $\eta_{2j-1} \in \CFD(X_K,-g)$, the horizontal chain from $\eta_{2j-1}$ to $\eta_{2j}$ provides the sequence
\[
\eta_{2j-1} \xrightarrow{D_3} \lambda^j_1 \xrightarrow{D_{23}} \dots \xrightarrow{D_{23}} \lambda^j_{\ell_j} \xrightarrow{D_2} \eta_{2j}.
\]
Thus, it remains to consider the case where $\ell_j=1$. Lemma \ref{lemma:length1} says that $\eta_{2j}$ is a linear combination of $\xi_2, \xi_4, \dots, \xi_{2n}$, along with $\xi_0$ provided that $\tau(K) = -g(K)+1$ and $g(K)>1$, and
only $D_{123}$ is nonzero on these elements (via corresponding vertical or unstable chains). Hence, the only $I$ such that $D_I \circ D_2 \circ D_3 |_{B_K}$ can be nonzero is $I=123$, as required.
\end{proof}

\begin{proposition} \label{prop:VK-D}
Let $K$ be a nontrivial knot with genus $g>0$ in an L-space homology sphere, and consider the subspace $V_K \subset \CFD(X_K)$ as described above.
\begin{enumerate}
\item \label{item:VKin}
The only possible nonzero sequences of coefficient maps into $V_K$ are $D_{123}$ and $D_1$. More precisely, if $\pi_V \circ D_{I_r} \circ \dots \circ D_{I_1} \ne 0$, then $r = 1$ and $I_1 = 123$ or $1$.

\item \label{item:VKout}
If the restriction of $D_{I_r} \circ \dots \circ D_{I_1}$ to $V_K$ is nontrivial, then $I_1 = 23$.
\end{enumerate}

The same is true for $V_K' \subset \CFD(X_K)'$.
\end{proposition}

\begin{proof}
By Theorem \ref{thm:cfkcfd}, the only coefficient maps whose image have nonzero projection to $V_K$ are $D_1$ and $D_{123}$. Furthermore, the only nonzero contribution to $\pi_V \circ D_1$ comes when $A(\tilde \xi_{2j}) = -g(K)$ and $k_j = 1$, in which case $D_1(\xi_{2j-1}) = \kappa^j_1$. It remains to verify that $\xi_{2j-1}$ has no incoming coefficient maps coming from the horizontal or unstable chains. If $\eta_{2i}$ has a nonzero $\xi_{2j-1}$ component, then $A(\eta_{2i}) = A(\xi_{2j-1})= -g(K)+1$ and $A(\eta_{2i-1}) = -g(K)$, so by Lemma \ref{lemma:length1}, $\eta_{2i}$ is in the span of $\xi_0, \xi_2, \dots, \xi_{2n}$, a contradiction. Likewise, if $\eta_0$ has a nonzero $\xi_{2j-1}$ component, then $\tau(K) = -A(\eta_0) = g(K)-1$ and $\epsilon(K) = -1$, so $g(K)>1$ by Lemma \ref{lemma:g1tau0}, hence $\tau(K)>0$. The unstable chain then gives $\eta_0$ an outgoing differential ($\eta_0 \xrightarrow{D_3} \mu_1$), not an incoming one. This concludes the proof of the first statement.

The second statement follows Proposition \ref{prop:BK-D} and the fact that $D_{123}|_{B_K} \co B_K \to V_K$ is an isomorphism.
\end{proof}

Next, we use the algorithm of Theorem \ref{thm:identity} to give analogous results for $\CFA(X_K)$. We view $\CFA(X_K)$ as having the same underlying vector space as $\CFD(X_K)$, with $\AA_\infty$ multiplications given by Theorem \ref{thm:identity}. We may then think of $B_K$, $V_K$, and $H_K$ as subspaces of $\CFA(X_K)$.

\begin{proposition} \label{prop:BK-A}
Let $K$ be a nontrivial knot with genus $g>0$ in an L-space homology sphere, and consider the subspace $B_K \subset \CFA(X_K)$ as described above.
\begin{enumerate}
\item \label{item:BKin-A}
Elements of $B_K$ have no incoming multiplications of any type. More precisely, for any $a_1, \dots, a_k \in \AA$, the composition $\pi_B \circ m_{k+1}(\cdot \otimes a_1 \otimes \dots \otimes a_k)$ is trivial.

\item \label{item:BKout-A}
If $I_1, \dots, I_r$ are elements of $\RR$ such that the restriction of $m_{r+1}(\cdot \otimes \rho_{I_1} \otimes \dots \otimes \rho_{I_r})$ to $B_K$ is nonzero, then:
\begin{enumerate}
    \item If $I_1 = 123$, then $r \ge 2$ and $I_2 = 2$.
    \item If $I_1 = 3$, then $r \ge 3$, $I_2 = 2$, and $I_3 = 1$ or $12$.
\end{enumerate}
\end{enumerate}
\end{proposition}

\begin{proof}
This proposition follows by applying Theorem \ref{thm:identity} to the results of Proposition \ref{prop:BK-D}. For any $I_1, \dots, I_r \in \RR$ with $I_1 \cdots I_r$ alternating and $\last(I_i) > \first(I_{i+1})$ for all $i$, we have
\[
m_{r+1}(\cdot \otimes \rho_{I_1} \otimes \dots \otimes \rho_{I_r}) = D_{\phi(J_j)} \circ \cdots \circ D_{\phi(J_1)},
\]
where $(J_1, \dots, J_j) = \Psi(I_1 \cdots I_r)$. If the restriction of $m_{r+1}(\cdot \otimes \rho_{I_1} \otimes \dots \otimes \rho_{I_r})$ to $B_K$ or $V_K$ is nonzero, the sequence $(\phi(J_1), \dots, \phi(J_j))$ must satisfy the conclusions of Proposition \ref{prop:BK-D}. Specifically:

\begin{itemize}
\item If $I_1 = 123$, then $\phi(J_1) = 3$ and $\phi(J_2) = 2$, so Proposition \ref{prop:BK-D} says that  $j>2$ and $J_3 = 123$. Hence $I_1 \cdots I_r= 12321 \cdots$, so $I_2 = 2$.

\item If $I_1 = 3$, then $\phi(J_1)$ begins with $1$, so Proposition \ref{prop:BK-D} says that $\phi(J_1)=123$ and $\phi(J_2) =23$ if $j>1$. Hence $I_1 \cdots I_r = 321$ or $32121\dots$, so $I_2 = 2$ and $I_3 = 1$ or $12$.
\end{itemize}
The proof is complete.
\end{proof}

A similar argument shows the following proposition.

\begin{proposition} \label{prop:VK-A}
Let $K$ be a nontrivial knot with genus $g>0$ in an L-space homology sphere, and consider the subspace $V_K \subset \CFA(X_K)$ as described above.
\begin{enumerate}
\item \label{item:VKin-A}
The only possible nonzero $\AA_\infty$ multiplications into $V_K$ are $m_2(\cdot \otimes \rho_3)$ and $m_4(\cdot \otimes \rho_3 \otimes \rho_2 \otimes \rho_1)$. More precisely, if $\pi_V \circ m_{r+1}(\cdot \otimes \rho_{I_1} \otimes \dots \otimes \rho_{I_r}) \ne 0$, then either $r = 1$ and $I_1 = 3$, or $r=3$ and $(I_1, I_2, I_3) = (3,2,1)$.

\item \label{item:VKout-A}
If the restriction of $m_{r+1}(\cdot \otimes \rho_{I_1} \otimes \dots \otimes \rho_{I_r})$ to $B_K$ is nonzero, then $I_1 = 2$. \qed
\end{enumerate}
\end{proposition}

\begin{proof}[Proof of Theorem \ref{thm:main}]
Let $K_1 \subset Y_1$ and $K_2 \subset Y_2$ be nontrivial knots in L-space homology spheres. The Alexander gradings on $\CFA(K_1)$ and $\CFD(K_2)'$ give a direct sum decomposition
of $\CFA(X_{K_1}) \boxtimes \CFD(X_{K_2})'$ as a vector space,
\[
\CFA(X_{K_1}) \boxtimes \CFD(X_{K_2})' = \bigoplus_{s \in \Z} C_s,
\]
where
\[
C_s = \bigoplus_{t \in \frac12 \Z} \CFA(X_{K_1}, t) \otimes_{\mathcal{I}} \CFD(X_{K_2}, s-t)'.
\]
Note that
\[
C_{-g(K_1)-g(K_2)} = B_{K_1} \otimes B_{K_2}'
\]
and
\begin{multline*}
C_{-g(K_1)-g(K_2)+1} = (V_{K_1} \otimes V_{K_2}') \oplus (V_{K_1} \otimes H_{K_2}') \oplus (H_{K_1} \otimes V_{K_2}') \oplus (H_{K_1} \otimes H_{K_2}') \\
\oplus (B_{K_1} \otimes \CFD(X_{K_2}, -g(K_2)+1)') \oplus (\CFA(X_{K_1}, -g(K_1)+1) \otimes B_{K_2}').
\end{multline*}
We claim that the direct summands $B = B_{K_1} \otimes B_{K_2}'$ and $V = V_{K_1} \otimes V_{K_2}'$, each of dimension $\dim \HFK(Y_1, K_1, -g(K_1)) \cdot \dim \HFK(H_2, K_2, -g(K_2))$, both survive in the homology of $\CFA(X_{K_1}) \boxtimes \CFD(X_{K_2})'$, which will prove that
\[
\dim HF(Y(K_1, K_2)) \ge 2 \dim \HFK(Y_1, K_1, -g(K_1)) \cdot \dim \HFK(H_2, K_2, -g(K_2)) \ge 2,
\]
as required.

To see that the differential on $\CFA(X_{K_1}) \boxtimes \CFD(X_{K_2})'$ is identically zero on $B$, we simply note that there do not exist $I_1, \dots, I_r \in \RR$ satisfying the conclusions of the second parts of Propositions \ref{prop:BK-D} and \ref{prop:BK-A} simultaneously. Thus, for any $x \in B_{K_1}$ and $y \in B_{K_2}'$,
\[
\partial^\boxtimes(x \otimes y) = \sum_{I_1, \dots, I_r \in \RR} m_{r+1}(x \otimes \rho_{I_1} \otimes \dots \otimes \rho_{I_r}) \otimes (D_{I_r} \circ \dots \circ D_{I_1})(y) = 0.
\]
(Here $m_{r+1}$ denotes an $\AA_\infty$ multiplication on $\CFA(X_{K_1})$, while $D_{I_1}, \dots, D_{I_r}$ denote coefficient maps on $\CFD(X_{K_2})'$.) Furthermore, the first parts of Propositions \ref{prop:BK-D} and \ref{prop:BK-A} imply that the composition of $\partial^\boxtimes$ with the projection onto $B$ coming from the direct sum decomposition is zero. Thus, $B$ survives in homology.

The proof for $V$ is similar, using Propositions \ref{prop:VK-D} and \ref{prop:VK-A}. Just as above, the restriction of $\partial^\boxtimes$ to $V$ vanishes. Furthermore, if $x \in \CFA(X_{K_1})$ and $y \in \CFD(X_{K_2})'$ are such that $\partial^\boxtimes(x \otimes y)$ has nontrivial projection to $V$, there must be $I_1, \dots, I_r$ that simultaneously satisfy the first parts of Propositions \ref{prop:VK-D} and \ref{prop:VK-A}, but clearly this is impossible.
\end{proof}

\section{Examples}\label{sec:examples}
Let $L$ and $R$ denote the left- and right-handed trefoils in $S^3$, respectively. For each of these knots, $\CFKm$ has a basis that is simultaneously horizontally and vertically simplified (up to permuting elements). The invariants $\CFD(X_{L})$ and $\CFD(X_{R})$ are as follows:

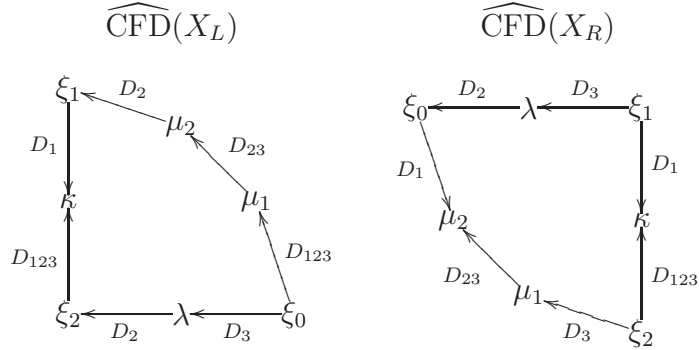
\[
\xy
(0,0)*{
    \xy
     (0,30)*{\xi_1}="x1";
     (0,0)*{\xi_2}="x2";
     (30,0)*{\xi_0}="x0";
     (0,15)*{\kappa}="k";
     (15,0)*{\lambda}="l";
     (25,15)*{\mu_1}="g1";
     (15,25)*{\mu_2}="g2";
     {\ar^{D_3} "x0";"l"};
     {\ar^{D_2} "l";"x2"};
     {\ar_{D_1} "x1";"k"};
     {\ar^{D_{123}} "x2";"k"};
     {\ar_{D_{123}} "x0";"g1"};
     {\ar_{D_{23}} "g1";"g2"};
     {\ar_{D_2} "g2";"x1"};
    \endxy};
(0,25)* {\CFD(X_{L})};
(50,0)*{
    \xy
     (0,30)*{\xi_0}="x0";
     (30,30)*{\xi_1}="x1";
     (30,0)*{\xi_2}="x2";
     (30,15)*{\kappa}="k";
     (15,30)*{\lambda}="l";
     (15,5)*{\mu_1}="g1";
     (5,15)*{\mu_2}="g2";
     {\ar_{D_3} "x1";"l"};
     {\ar_{D_2} "l";"x0"};
     {\ar^{D_1} "x1";"k"};
     {\ar_{D_{123}} "x2";"k"};
     {\ar^{D_3} "x2";"g1"};
     {\ar^{D_{23}} "g1";"g2"};
     {\ar_{D_1} "x0";"g2"};
    \endxy};
(50,25)* {\CFD(X_{R})};
\endxy
\]
According to Theorem \ref{thm:identity}, $\CFA(X_{R})$ is as follows (using capital Greek letters to avoid confusion when we take tensor products below):
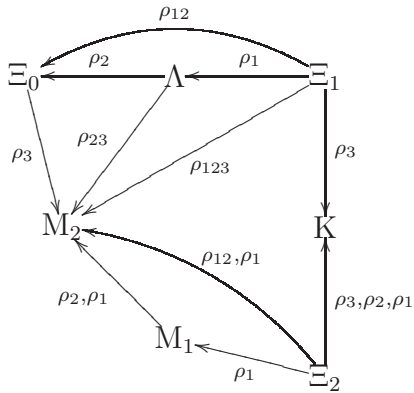
\[
\xy
 (0,40)*{\Xi_0}="X0";
 (40,40)*{\Xi_1}="X1";
 (40,0)*{\Xi_2}="X2";
 (40,20)*{\Kappa}="K";
 (20,40)*{\Lambda}="L";
 (20,5)*{\Mu_1}="G1";
 (5,20)*{\Mu_2}="G2";
 {\ar_{\rho_1} "X1";"L"};
 {\ar_{\rho_2} "L";"X0"};
 {\ar^{\rho_3} "X1";"K"};
 {\ar_{\rho_3, \rho_2, \rho_1} "X2";"K"};
 {\ar^{\rho_1} "X2";"G1"};
 {\ar^{\rho_2, \rho_1} "G1";"G2"};
 {\ar_{\rho_3} "X0";"G2"};
 {\ar_{\rho_{23}} "L";"G2"};
 {\ar^{\rho_{123}} "X1";"G2"};
 {\ar@/_1.5pc/_{\rho_{12}} "X1";"X0"};
 {\ar@/_1pc/_{\rho_{12},\rho_1} "X2";"G2"};
\endxy
\]

We may use these results to compute the tensor product complexes $\CFA(X_{R}) \boxtimes \CFD(X_{L})$ and $\CFA(X_{R}) \boxtimes \CFD(X_{R})$, illustrated in Figures \ref{fig:RHTLHT} and \ref{fig:RHTRHT}. In each of these figures, the two homology classes provided by the proof of Theorem \ref{thm:main} are indicated in boldface.

From these complexes, it is easy to verify that
\[
\dim \HF(Y(R, L)) = \dim H_*(\CFA(X_{R}) \boxtimes \CFD(X_{L}) ) = 9
\]
and
\[
\dim \HF(Y(R, R)) = \dim H_*(\CFA(X_{R}) \boxtimes \CFD(X_{R}) ) = 7.
\]
Since $Y(L,L) = -Y(R, R)$, we also have
\[
\dim \HF(Y(L, L)) = 7.
\]
The reader is encouraged verify these results in another way by computing $\CFA(X_{L})$ and evaluating its box tensor product with $\CFD(X_L)$ and $\CFD(X_R)$.

\begin{figure}
\[
\xy
 (-7,7)*{\Xi_0 \xi_1}="X0x1";
 (-7,-7)*{\Xi_0 \xi_2}="X0x2";
 (7,-7)*{\Xi_0 \xi_0}="X0x0";
 (40,10)*{\Lambda \mu_2}="Lm2";
 (30,0)*{\Lambda \kappa}="Lk";
 (40,-10)*{\Lambda \lambda}="Ll";
 (50,0)*{\Lambda \mu_1}="Lm1";
 (73,7)*{\Xi_1 \xi_1}="X1x1";
 (73,-7)*{\Xi_1 \xi_2}="X1x2";
 (87,-7)*{\Xi_1 \xi_0}="X1x0";
 (80,-30)*{\Kappa \mu_2}="Km2";
 (70,-40)*{\Kappa \kappa}="Kk";
 (80,-50)*{\Kappa \lambda}="Kl";
 (90,-40)*{\bm{\Kappa \mu_1}}="Km1";
 (73,-73)*{\Xi_2 \xi_1}="X2x1";
 (73,-87)*{\Xi_2 \xi_2}="X2x2";
 (87,-87)*{\bm{\Xi_2 \xi_0}}="X2x0";
 (40,-60)*{\Mu_1 \mu_2}="M1m2";
 (30,-70)*{\Mu_1 \kappa}="M1k";
 (40,-80)*{\Mu_1 \lambda}="M1l";
 (50,-70)*{\Mu_1 \mu_1}="M1m1";
 (10,-30)*{\Mu_2 \mu_2}="M2m2";
 (0,-40)*{\Mu_2 \kappa}="M2k";
 (10,-50)*{\Mu_2 \lambda}="M2l";
 (20,-40)*{\Mu_2 \mu_1}="M2m1";
 {\ar "X1x1";"Lk"}; 
 {\ar@/^1pc/ "X2x1";"M1k"}; 
 {\ar "Lm2";"X0x1"}; 
 {\ar@/^1pc/ "Ll";"X0x2"}; 
 {\ar@/^0.5pc/ "X1x0";"Kl"}; 
 {\ar@/_1pc/ "X0x0";"M2l"}; 
 {\ar@/_1pc/ "Lm1";"M2m2"}; 
 {\ar "X1x0";"M2m1"}; 
 {\ar "X1x2";"M2k"}; 
 {\ar "X1x2";"M2k"}; 
 {\ar "M1m2";"M2k"}; 
\endxy
\]
\caption{The tensor product complex $\CFA(X_{R}) \boxtimes \CFD(X_{L})$.} \label{fig:RHTLHT}
\end{figure}

\begin{figure}
\[
\xy
 (-7,7)*{\Xi_0 \xi_0}="X0x0";
 (7,7)*{\Xi_0 \xi_1}="X0x1";
 (7,-7)*{\Xi_0 \xi_2}="X0x2";
 (40,10)*{\Lambda \lambda}="Ll";
 (30,0)*{\Lambda \mu_2}="Lm2";
 (40,-10)*{\Lambda \mu_1}="Lm1";
 (50,0)*{\Lambda \kappa}="Lk";
 (73,7)*{\Xi_1 \xi_0}="X1x0";
 (87,7)*{\Xi_1 \xi_1}="X1x1";
 (87,-7)*{\Xi_1 \xi_2}="X1x2";
 (80,-30)*{\Kappa \lambda}="Kl";
 (70,-40)*{\Kappa \mu_2}="Km2";
 (80,-50)*{\Kappa \mu_1}="Km1";
 (90,-40)*{\bm{\Kappa \kappa}}="Kk";
 (73,-73)*{\Xi_2 \xi_0}="X2x0";
 (87,-73)*{\Xi_2 \xi_1}="X2x1";
 (87,-87)*{\bm{\Xi_2 \xi_2}}="X2x2";
 (40,-60)*{\Mu_1 \lambda}="M1l";
 (30,-70)*{\Mu_1 \mu_2}="M1m2";
 (40,-80)*{\Mu_1 \mu_1}="M1m1";
 (50,-70)*{\Mu_1 \kappa}="M1k";
 (10,-30)*{\Mu_2 \lambda}="M2l";
 (0,-40)*{\Mu_2 \mu_2}="M2m2";
 (10,-50)*{\Mu_2 \mu_1}="M2m1";
 (20,-40)*{\Mu_2 \kappa}="M2k";
 {\ar@/^0.5pc/ "X1x1";"Lk"}; 
 {\ar "X1x0";"Lm2"}; 
 {\ar@/_1pc/ "X2x1";"M1k"}; 
 {\ar@/^1pc/ "X2x0";"M1m2"}; 
 {\ar@/_1pc/ "Ll";"X0x0"}; 
 {\ar@/_1pc/ "X1x1";"Kl"}; 
 {\ar@/^0.5pc/ "X1x2";"Km1"}; 
 {\ar@/^1pc/ "X0x1";"M2l"}; 
 {\ar@/_1pc/ "X0x2";"M2m1"}; 
 {\ar@/^1pc/ "Lm1";"M2m2"}; 
 {\ar "X1x2";"M2k"}; 
 {\ar "M1l";"M2m2"}; 
 {\ar "X2x1";"Km2"}; 
\endxy
\]
\caption{The tensor product complex $\CFA(X_{R}) \boxtimes \CFD(X_{R})$.} \label{fig:RHTRHT}
\end{figure}

\section{Future Directions}

We conclude by discussing the prospects for generalizing Theorem \ref{thm:main} to manifolds obtained by splicing knots in arbitrary homology spheres, which would prove Conjecture \ref{conj:torus}. If $K$ is a knot in a homology sphere $Y$, the proof of Theorem \ref{thm:cfkcfd} given in \cite{LOTBordered} can be adapted to give a description of $\CFD(X_K)$ in terms of $\CFKm(Y,K)$, with multiple unstable chains when $Y$ is not an L-space. However, the structure of the unstable chains depends on the isomorphism induced on homology by a certain chain homotopy equivalence
\[
J\co (C^h, \partial^h) \to (C^v, \partial^v)
\]
that arises in the course of the proof, and this isomorphism is not \emph{a priori} determined merely by $\CFKm(X,K)$. Furthermore, even though $(C^h, \partial^h)$ and $(C^v, \partial^v)$ are filtered chain homotopy equivalent, the map $J$ need not be a filtered chain homotopy equivalence. In particular, an unstable chain may connect a horizontal generator $\eta_0$ and a vertical generator $\xi_0$ with the property that $A(\eta_0) \ne -A(\xi_0)$.

As a result, Propositions \ref{prop:BK-D} through \ref{prop:VK-A} no longer hold when $Y$ is not an L-space. For example, let $Y$ be the manifold obtained by $+1$ surgery on the left-handed trefoil $L$ (i.e., the Brieskorn sphere $-\Sigma(2,3,7)$), and let $K$ be the core of the surgery torus. Note that $X_K = X_L$ as smooth manifolds with boundary, but the parametrization $\phi_K$ differs from $\phi_L$ by a longitudinal Dehn twist. Thus,
\[
\CFD(X_K, \phi_K) \simeq \CFDA(\tau_\lambda^{-1}) \boxtimes \CFD(X_L, \phi_L),
\]
where $\CFDA(\tau_\lambda^{-1})$ is one of the Dehn twist bimodules computed in \cite[Section 10.2]{LOTBimodules}. By evaluating this tensor product and simplifying, the reader may verify that $\CFD(X_K, \phi_K)$ has the following form:
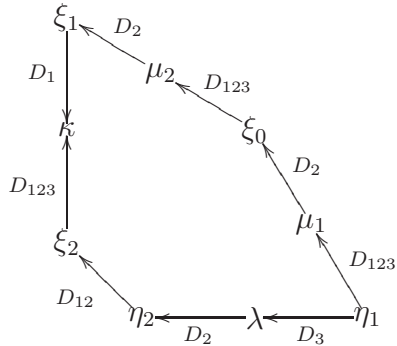
\[
\xy
 (0,30)*{\xi_1}="x1";
 (0,0)*{\xi_2}="x2";
 (10,-10)*{\eta_2}="e2";
 (40,-10)*{\eta_1}="e1";
 (25,15)*{\xi_0}="x0";
 (0,15)*{\kappa}="k";
 (25,-10)*{\lambda}="l";
 (32.5,2.5)*{\mu_1}="m1";
 (12.5,22.5)*{\mu_2}="m2";
 {\ar^{D_3} "e1";"l"};
 {\ar^{D_2} "l";"e2"};
 {\ar_{D_1} "x1";"k"};
 {\ar^{D_{123}} "x2";"k"};
 {\ar_{D_{123}} "e1";"m1"};
 {\ar_{D_2} "m2";"x1"};
 {\ar^{D_{12}} "e2";"x2"};
 {\ar_{D_2} "m1";"x0"};
 {\ar_{D_{123}} "x0";"m2"};
 \endxy
\]
Here, $\eta_0$, $\eta_2$, and $\xi_0$ are the generators of vertical homology, and $\xi_0$, $\xi_1$, and $\xi_2$ are the generators of horizontal homology. The only generator in Alexander grading $-1$ is $\eta_1$. Notice that $D_2 \circ D_{123}(\eta_1)$ and $D_{12} \circ D_2 \circ D_3(\eta_1)$ are both nonzero (and distinct), contrary to Proposition \ref{prop:BK-D}. Furthermore, by Theorem \ref{thm:identity}, the corresponding generator in $\CFA(X_K)$ has outgoing $m_4(\cdot \otimes \rho_3 \otimes \rho_2 \otimes \rho_{12})$ and $m_3(\cdot \otimes \rho_{123} \otimes \rho_2)$ multiplications, contrary to Proposition \ref{prop:BK-A}. Therefore, when $K_1$ and $K_2$ are knots in arbitrary homology spheres, the subgroup
\[
B_{K_1} \otimes B_{K_2} \subset \CFA(X_{K_1}) \boxtimes \CFD(X_{K_2})
\]
does not necessarily survive in homology, unlike in our proof of Theorem \ref{thm:main}. A different strategy will thus be required for a proof of Conjecture \ref{conj:torus}.

\bibliographystyle{amsplain}
\bibliography{bibliography}

\end{document}